\documentclass[ejsv2]{imsart}

\RequirePackage{amsthm,amsmath,amsfonts}
\RequirePackage[numbers]{natbib}
\RequirePackage[colorlinks,citecolor=blue,urlcolor=blue]{hyperref}
\RequirePackage{graphicx}

\usepackage[utf8]{inputenc}
\usepackage{dsfont}
\usepackage{mathrsfs}
\usepackage{url}
\usepackage{bbm}
\usepackage{verbatim}
\usepackage{fancybox}
\usepackage{enumitem}
\usepackage{eurosym}
\usepackage{cleveref}
\usepackage{autonum}
\usepackage[lined,longend,algoruled,linesnumbered,boxed]{algorithm2e}
\DontPrintSemicolon
\usepackage{algorithmic}

\usepackage[english]{babel}

\arxiv{2407.14778}
\startlocaldefs

\def\d{\delta}

\def\t{\theta}
\def\T{\Theta}
\def\i{\iota}

\def\s{\sigma}
\def\S{\Sigma}

\def\ZZ{\mathbb Z}

\def\RR{\mathbb R}

\def\calT{\mathcal T}

\def\esp{{\mathbf E}_{\bs{\t},\bs{\S}}}
\def\var{\mathbf{Var}}
\def\cov{\mathbf{Cov}}
\def\prob{{\mathbf P}_{\bs{\t},\bs{\S}}}
\def\calN{\mathcal N}

\def\trace{\mathbf{tr}}
\theoremstyle{plain}
\newtheorem{theorem}{Theorem}
\newtheorem{lemma}{Lemma}
\newtheorem{proposition}{Proposition}
\newtheorem{corollary}{Corollary}
\newtheorem*{theorem*}{Theorem}
\newtheorem*{lemma*}{Lemma}
\newtheorem*{proposition*}{Proposition}
\newtheorem*{corollary*}{Corollary}
\theoremstyle{remark}
\newtheorem{remark}{Remark}
\newtheorem*{remark*}{Remark}

\newtheorem*{note*}{Note}
\theoremstyle{definition}

\newtheorem*{definition*}{Definition}
\newtheorem{condition}{Assumption}

\def\bt{\boldsymbol \t}

\def\bY{{\mathbf Y}}

\def\beps{{\boldsymbol \epsilon}}

\def\one{\mathbbm{1}}

\def\cS{\mathcal{S}}
\def\cN{\mathcal{N}}
\def\cE{\mathcal{E}}
\def\varz{z}
\newcommand{\bs}[1]{{\boldsymbol #1}}

\newcommand{\pth}[1]{\left( #1 \right)}
\newcommand{\qth}[1]{\left[ #1 \right]}
\newcommand{\sth}[1]{\left\{ #1 \right\}}
\newcommand{\abs}[1]{\left| #1 \right|}

\newcommand{\nbyw}[1]{{\sf\color{red}[YW: #1]}}
\def\beps{\bs{\varepsilon}}
\def\rev{\color{black}}

\newcommand{\vertiii}[1]{{\left\vert\kern-0.25ex\left\vert\kern-0.25ex\left\vert #1
		\right\vert\kern-0.25ex\right\vert\kern-0.25ex\right\vert}}

\makeatletter
\newcommand{\customlabel}[2]{%
	\protected@write \@auxout {}{\string \newlabel {#1}{{#2}{\thepage}{#2}{#1}{}} }%
	\hypertarget{#1}{}
}
\makeatother

\theoremstyle{remark}

\endlocaldefs

\begin{document}
\begin{frontmatter}
\title{Minimax estimation of functionals in sparse vector model with correlated observations}
\runtitle{Minimax estimation with correlated observations}

\begin{aug}
\author[A]{\fnms{Yuhao}~\snm{Wang}\ead[label=e1]{yuhaow@tsinghua.edu.cn}},
\author[B]{\fnms{Pengkun}~\snm{Yang}\ead[label=e2]{yangpengkun@tsinghua.edu.cn}}
\and
\author[C]{\fnms{Alexandre}~\snm{Tsybakov}\ead[label=e3]{Alexandre.Tsybakov@ensae.fr}}
\address[A]{Institute for Interdisciplinary Information Sciences, Tsinghua University and Shanghai Qi Zhi Institute\printead[presep={,\ }]{e1}}

\address[B]{Department of Statistics and Data Science, Tsinghua University\printead[presep={,\ }]{e2}}

\address[C]{CREST, ENSAE, Institut Polytechnique de Paris\printead[presep={,\ }]{e3}}
\runauthor{Y. Wang et al.}
\end{aug}

\begin{abstract}
We consider the observations of an unknown $s$-sparse vector $\bs{\t}$ corrupted by Gaussian noise with zero mean and unknown covariance matrix $\bs{\S}$. We propose minimax optimal methods of estimating the $\ell_2$ norm of  $\bs{\t}$ and testing the hypothesis $H_0: \bs{\t}=0$ against sparse alternatives when only partial information about $\bs{\S}$ is available, such as an upper bound on its Frobenius norm and the values of its diagonal entries to within an unknown scaling factor. We show that the minimax rates of the estimation and testing are leveraged not by the dimension of the problem but by the value of the Frobenius norm of $\bs{\S}$.
\end{abstract}

\begin{keyword}[class=MSC]
	\kwd[Primary ]{62J05}
	\kwd{62G05}
\end{keyword}

\begin{keyword}
	\kwd{Correlated noise}
	\kwd{Gaussian sequence model}
	\kwd{Nonasymptotic minimax estimation}
	\kwd{Sparsity}
	\kwd{Unknown noise level}
\end{keyword}

\end{frontmatter}

\section{Introduction}\label{sec:intro}

In this paper, we consider the Gaussian vector model 
\begin{equation}
	\label{eq:model}
	\bs{Y} = \bs{\t} + \s \bs{\varepsilon},
\end{equation}
where $\bs{\varepsilon} \sim \cN(\bs{0}, \bs{\S})$ follows a $d$-dimensional zero-mean Gaussian distribution with  covariance matrix $\bs{\S}$, and $\s > 0$ is the noise level. We assume throughout that the diagonal elements $\sigma_i^2$ of matrix $\bs{\S}$ are known (the main example is $\sigma_i^2=1$ for all $i$, that is, we deal with a normalized matrix), while its off-diagonal elements are not, and $\sigma$ is either known or unknown. Based on the data $ \bs{Y}=(Y_1,\dots,Y_d)$, we consider the problem of estimating the quadratic functional and the $\ell_2$-norm of the parameter $\bs{\t} = (\t_1, \dots, \t_d)$:
\[
Q(\bs{\t})=\sum_{i=1}^d \theta_i^2,\quad \|\bs{\t}\|_2 = \sqrt{Q(\bs{\t})}.
\]
We assume that $\bs{\t}$ belongs to the class  $B_0(s, d)$ of all vectors $\bs{\t} \in \mathbb{R}^d$ such that $\|\bs{\t}\|_0 \le s$ ($s$-sparse vectors), where $s$ is an integer in $[1,d]$.
Here, $\|\bs{\t}\|_0$ denotes the number of non-zero components of $\bs{\t}$. When $\sigma$ is unknown, we also study the problem of estimating $\sigma$.



Most of the existing literature on estimation of functionals and on the related problem of testing under sparse vector models is focused on the case of uncorrelated data ($\bs{\S}$ is a diagonal matrix). However, in many applications such as time series analysis one needs to deal with dependent observations. 
In this paper, we study the setting where only partial information on $\bs{\S}$ is available, namely, either the Frobenius norm of $\bs{\S}$ or an upper bound is known. We propose estimators of $Q(\bs{\t})$ and $\|\bs{\t}\|_2$ that are adaptive to the noise level and are rate optimal in a minimax sense for $\|\bs{\t}\|_2$. 
As a consequence, we construct minimax optimal tests for signal detection. Furthermore, we propose minimax optimal estimators of the noise level $\sigma$.

When $\bs{\S}=\bs{I}$, where $\bs{I}$ is the identity matrix, and $\s$ is known,~\citet{CCT2017} proved that the rate of the minimax quadratic risk for estimating the $\ell_2$-norm $\|\bs{\t}\|_2$ of a $s$-sparse vector is equal to $\s^2 \phi(s, d)$, where for any $t>0$,
\begin{equation}\label{cct-rate}
	\phi(s,t): = \left\{\begin{array}{lc}
		s \log(1 + t  / s^2) & \textrm{if}\; s \le \sqrt{t}, \\
		\sqrt{t} & \textrm{otherwise}.
	\end{array}\right.  
\end{equation}
Moreover, \citet{CCT2017} derived the minimax rates for estimation of $Q(\bs{\t})$, and for testing the hypothesis $H_0: \bs{\t} = \bf{0}$ against sparse alternatives separated from $\bf{0}$ in the $\ell_2$-norm (improving the result of \citet{baraud} by matching the upper and lower bounds). Estimation of $\|\bs{\t}\|_2$ is related to such a testing problem in the sense that upper bounds on the risks for the problem of estimating $\|\bs{\t}\|_2$ imply upper bounds for testing, while lower bounds for the testing problem imply lower bounds for estimating $\|\bs{\t}\|_2$.  
\citet{comminges2021adaptive} considered the model \eqref{eq:model} with unknown $\s$ and uncorrelated noise (not necessarily Gaussian), and derived the minimax rates of estimating the functionals $Q(\bs{\t})$ and $\|\bs{\t}\|_2$ as well as of the noise level $\sigma$ on the class of $s$-sparse vectors. In particular, \citet{comminges2021adaptive} proved that, when the noise is i.i.d. Gaussian and $\s$ is unknown, the rate of the minimax quadratic risk for estimating the $\ell_2$-norm of a $s$-sparse vector is equal to $\s^2 \phi^*(s, d)$, where for any $t>0$,
\begin{equation}\label{rate-comminges-unknown sigma}
	\phi^*(s, t) = \left\{
	\begin{array}{lc}
		s \log(1 + t / s^2) & \textrm{if}\; s \le \sqrt{t}, \\
		\frac{s}{1 \vee \log(s / \sqrt{t})} & \textrm{otherwise.}
	\end{array}
	\right.  
\end{equation}
\citet{carpentier2019minimax,CCCTW22} further extended these results to the model of 
high-dimensional sparse regression with i.i.d. noise. Model \eqref{eq:model} in the case of  diagonal matrix $\bf{\S}$ with known entries and $\sigma=1$ was analyzed in \citet{LLM2012,chhor2022sparse}. Under such particular conditions, \citet{LLM2012} studied the minimax rates for the problem of testing mentioned above. In  \citet{chhor2022sparse} these results are extended to testing of $H_0: \bs{\t} = \bf{0}$ against sparse alternatives separated from $\bf{0}$ in the $\ell_p$-norm with $1\le p<\infty$ and further to estimation of the corresponding norms.
Another recent development due to \citet{cannone2022robust} shows that introducing an arbitrary contamination in model \eqref{eq:model} with $\bs{\S}=\bs{I}$ leads to slower minimax rates (for the problem of testing with $\ell_2$-separation) than those given in \cite{CCT2017}.
Finally, there exist works allowing for correlated noise, again focused on the problem of testing, see \citet{hall2009jin}, \citet{Gao2021change-point}, \citet{kotekal2021minimax}. The paper \cite{Gao2021change-point} deals with a change-point framework and establishes the minimax rates of testing for the particular case $s=d$ (no sparsity) assuming that matrix $\bs{\S}$ is fixed and positive definite. 
In \cite{kotekal2021minimax} the authors consider model \eqref{eq:model} with $\s=1$, equicorrelation matrix $\bs{\S} = (1 - \gamma)\bs{I} + \gamma \bs{1} \bs{1}^\top$, where $\gamma>0$, and $\bs{1}$ is a $d$-dimensional column vector with all entries equal to $1$, and establish the minimax rates of testing the hypothesis $H_0: \bs{\t} = \bf{0}$  against sparse alternatives separated from $\bf{0}$ in the $\ell_2$-norm assuming that $\gamma$ is known. 
{\rev They show that, for some choices of~$\gamma$, the detection boundary can be smaller than its value when the noise is uncorrelated.} In a more recent work  \cite{kotekal2023sparsity}, the same authors find the minimax rates of estimating linear functionals under the equicorrelated Gaussian noise. {\rev A different asymptotic instance of the same testing problem is considered in \citet{hall2009jin}, where it is assumed that all non-zero components of $\theta$ have the same value $r\sqrt{\log d}$ for some $r>0$, the sparsity parameter satisfies $s=d^a$ for some $a\in(0,1/2)$, and the positions of non-zero components are randomly generated. In this model, it is proved in \citet{hall2009jin} that if $\bs{\S}$ is a Toeplitz matrix, then the asymptotic (as $d\to \infty$) detection boundary can be smaller than its value when the noise is uncorrelated. The improvement is not in the rate but in the constant factor.}

In summary, the existing literature dealing with our non-asymptotic setting
(mostly focused on the problem of testing rather than estimation of functionals) allows one to tackle the problem under sparsity when $\bs{\S}$ is either a diagonal matrix with known entries or the equicorrelation matrix, and in the latter case all parameters of the model except $\bs{\t}$ need to be known. However, under general correlation structure, the question of how to estimate functionals of a sparse vector $\bs{\t}$ in a minimax optimal way and how to construct the associated tests remains open. In this paper, we focus on the problem of estimating the functionals $Q(\bs{\t})$ and $\|\bs{\t}\|_2$ and we consider the general setting, with arbitrary structure of $\bs{\S}$ and with the noise level $\s$ that can be either known or unknown.
Moreover, the estimators that we propose do not require knowing the off-diagonal entries of $\bs{\S}$. They need only partial information about $\bs{\S}$ such as the relative scale of the diagonal entries, the value of its Frobenius norm  or an upper bound on it. We construct estimators of $\|\bs{\t}\|_2$ that are optimal in a minimax sense and are adaptive to the noise level $\sigma$. We also derive the associated minimax rates of testing. The minimax rates that we obtain for estimating $\|\bs{\t}\|_2$ under known and unknown $\sigma$ are of the form \eqref{cct-rate} and \eqref{rate-comminges-unknown sigma}, respectively, with $t=\|\bs{\S}\|_F^2$, where $\|\bs{\S}\|_F$ is the Frobenius norm of~$\bs{\S}$. If $\bs{\S}=\bs{I}$ we recover the rates $\s^2 \phi(s, d)$ and $\s^2 \phi^*(s, d)$  obtained in \cite{CCT2017,comminges2021adaptive}. For general~$\bs{\S}$, we arrive at an interesting conclusion that the possibility of faster or slower rates is leveraged not by the dimension $d$ but by the value of the Frobenius norm of $\bs{\S}$.

Obtaining the results for the model with correlated noise is challenging. The tools used in~\cite{comminges2021adaptive} for the setting with unknown $\s$, including exponential inequalities and sample splitting to obtain multiple independent estimators,  break down when the observations are correlated. We develop new proof techniques allowing us to
obtain a slight improvement upon~\cite{comminges2021adaptive} even in the case of i.i.d. observations, {\rev which consists in the fact that} our estimator does not need sample splitting.

{\rev This paper is organized as follows. In~\Cref{sec:known-sigma}, we discuss estimation of $Q(\bt)$ and $\|\bt\|_2$ when~$\s$ is known. In~\Cref{sec:sigma}, we introduce an estimator of $\s$. \Cref{sec:unknown-sigma} deals with the problem of estimating $\|\bt\|_2$ when $\s$ is unknown. In~\Cref{sec:tests}, we consider testing of the null hypothesis $\|\bt\|_2 = 0$. \Cref{sec:proof} is devoted to the proofs.}

{\bf Notation.}
Let $[n]:=\{1,\dots,n\}$.
For $1\le q\le \infty$, let $\|\cdot\|_q$ denote the $\ell_q$-norm in $\RR^d$, and let $B_q(r):=\{x\in\RR^d :\|x\|_q\le r\}$ be the $\ell_q$-ball of radius $r$.
We denote by $\Phi$ the cumulative distribution function of the standard Gaussian distribution.
For $x,y\in\RR$, we set $x\wedge y := \min\{x,y\}$ and $x\vee y:=\max\{x,y\}$.
We use standard big-$O$ notations, e.g., for two positive sequences $\{a_n\}$ and $\{b_n\}$, we write $a_n=O(b_n)$ or $a_n\lesssim b_n$ if $a_n\le C b_n$ for some absolute constant $C>0$; $a_n=\Omega(b_n)$ or $a_n\gtrsim b_n$ if $b_n=O(a_n)$; $a_n=\Theta(b_n)$ or $a_n\asymp b_n$ if $a_n=O(b_n)$ and $b_n=O(a_n)$.
% ...
Given a matrix $A\in\RR^{d\times d}$ and $\mathcal{I}, \mathcal{J}\subseteq [d]$, we denote by $A_{\mathcal{I}, \mathcal{J}}$ the submatrix of $A$ with rows from $\mathcal{I}$ columns from $\mathcal{J}$, by $A_{ij}$ the entries of $A$, and by $\|A\|_F$, $\lambda_{\max}(A)$, and $\trace[A]$, the Frobenius norm, the maximal eigenvalue, and the trace of $A$, respectively. We denote by $\esp$ the expectation with respect to the probability measure $\prob$ of random vector $\bs{Y}$ satisfying \eqref{eq:model}. We denote by $c,C,C'$ absolute positive constants that can vary from line to line.

\section{Estimating the quadratic functional and the $\ell_2$-norm when $\sigma$ is known}
\label{sec:known-sigma}

In this section, we consider the case of known noise level $\s$. Throughout the paper, we assume the following condition on the diagonal entries of the covariance matrix.
\begin{condition} \label{as:var0}
	The diagonal elements $\sigma_{i}^2, i=1,\dots,d,$ of the covariance matrix  matrix $\bs{\S}$ satisfy $\max_{1\le i\le d}\sigma_{i}^2 \le 1 $. 
\end{condition}
Assumption \ref{as:var0} is just a normalization condition. It assures that $\sigma$ can be considered as the noise level. As $\bs{\S}$ is a covariance matrix Assumption \ref{as:var0} implies that all entries of $\bs{\S}$ are bounded by 1. 
Note also that we can replace the condition $\max_{1\le i\le d}\sigma_{i}^2 \le 1 $ by $\max_{1\le i\le d}\sigma_{i}^2 \le c $, where $c>0$ is an absolute constant, with no incidence on the results.

Assume first that the Frobenius norm of the covariance matrix $\bs{\S}$ is given. Then we define the estimator $\hat Q$ of the quadratic functional $Q(\bs{\t})$  as
\begin{equation}
	\Hat{Q} =
	\begin{cases}
		\sum_{i=1}^d (Y_i^2 - \s^2 \s_{i}^2\beta_s) \one\{|Y_i| >  \s \s_{i} \tau \}, & s \le \|\bs{\S}\|_F,\\
		\sum_{i=1}^d (Y_i^2 - \s^2 \s_{i}^2), & \text{otherwise},
	\end{cases}
\end{equation}
where $\beta_s := {\mathbf E}\left[Z^2 | \, |Z| \ge \tau \right]$ with $Z\sim \mathcal{N}(0,1)$, and $\tau = 3 \sqrt{\log ( 1 + \|\bs{\S}\|_F^2 / s^2)}$. The associated estimator of the $\ell_2$-norm $\|\bs{\t}\|_2$ is defined as
\[
\hat{N} = \sqrt{|\hat{Q}|}.
\]
Set
\[
\psi(s,\bs{\S}) := \phi(s,\|\bs{\S}\|_F^2) = \left\{\begin{array}{lc}
	s \log(1 + \|\bs{\S}\|_F^2  / s^2) & \textrm{if}\; s \le \|\bs{\S}\|_F, \\
	\|\bs{\S}\|_F & \textrm{otherwise,}
\end{array}\right.
\]
and
\[
\bar\psi(s,\bs{\S}) := \lambda_{\max}(\bs{\S}) \wedge s. 
	\]
	Let $\cS := \{i : \t_i \neq 0\}$ denote the support of $\bs{\t}$. The following theorem provides upper bounds for the mean square risks of the estimators $\hat{Q}$ and $\hat{N}$.
	
	\begin{theorem}\label{thm:knownvar}
		There exists an absolute constant $C>0$ such that, for any $\bs{\t}\in B_0(s, d)$ and any $\bs{\S}$ satisfying Assumption \ref{as:var0},
		\begin{equation}\label{thmeq:hatq}
			\esp[(\hat{Q} - \|\bs{\t}\|_2^2)^2] \le C \left(\s^2  \|\bs{\t}\|_2^2 \lambda_{\max}(\bs{\S_{\cS, \cS}})+ \s^4 \psi^2(s, \bs{\S})\right).
		\end{equation}
		Furthermore, for any $\bs{\S}$ satisfying Assumption \ref{as:var0},
		\begin{equation}\label{thmeq:hatq0}
			\sup_{\bs{\t} \in B_0(s, d) \cap B_2(r)}\esp[(\hat{Q} - \|\bs{\t}\|_2^2)^2] \le C \left(\s^2  r^2 \bar\psi(s,\bs{\S})+ \s^4 \psi^2(s, \bs{\S})\right)
		\end{equation}
		and
		\begin{equation}\label{thmeq:hatn}
			\sup_{\bs{\t} \in B_0(s, d)}\esp[(\hat{N} - \|\bs{\t}\|_2)^2] \le C \s^2 \psi(s,\bs{\S}).
		\end{equation}
	\end{theorem}
	It is interesting to compare \eqref{thmeq:hatn} to the bound for the case $\bs{\S}=\bs{I}$  \cite{CCT2017}, that is, $\s^2\phi(s,d)$, where $\phi(s,d)$ is defined in \eqref{cct-rate}. Since $\bs{\S}$ is a covariance matrix, Assumption \ref{as:var0} implies that $\|\bs{\S}\|_F$ is at most of order $d$. Particularly, in the extreme case $\|\bs{\S}\|_F\asymp d$, there is no elbow in the rate and $\psi(s,\bs{\S})\asymp s\log(1 + d^2/s^2)$ for all $1\le s\le d$. Thus, we get a substantial deterioration in the zone $s>\sqrt{d}$ compared to the rate $\s^2\phi(s,d)$ with uncorrelated data. On the contrary, for the other extreme case $\|\bs{\S}\|_F\asymp \sqrt{d}$, we have $\psi(s,\bs{\S})\asymp \phi(s,d)$ and we recover the rate obtained in \cite{CCT2017} for $\bs{\S}=\bs{I}$. Finally, if $\bs{\S}$ is a very sparse matrix, with very small diagonal elements (or some diagonal elements equal to 0, that is, no noise in some components of $\bs{\t}$), then the rate can be further improved.
	
	{\rev If only an upper bound on $\|\bs{\S}\|_F$ is available, we can modify the definition of $\hat{N}$ by using this upper bound in place of $\|\bs{\S}\|_F$. We now show that such an estimator is minimax optimal in the following sense. Consider the class of parameters
		\begin{align}
			\T(s, d, \rho) := \{(\bs{\t}, \bs{\S}) \in \RR^d \times \RR^{d \times d}: & \;\bs{\t} \in B_0(s, d), \; \|\bs{\S}\|_F \le \rho \\
			& \;\;\textrm{and} \; \bs{\S}_{ii}=1 \ \textrm{for} \ i=1,\dots,d \},
		\end{align}
}
	where $\rho\in (0,d]$ is a given value. Note that it makes no sense to consider $\rho>d$, since the bound $\|\bs{\S}\|_F \le d$ is trivially satisfied if $\bs{\S}_{ii}=1$ for $i=1,\dots,d$. The following lower bound holds.
	\begin{theorem}\label{thm:lower-known-sigma-rho}
		Let $w:[0,\infty)\to[0,\infty)$ be a monotone non-decreasing function such that $w(0)=0$ and $w\not\equiv 0$. Then there exists a constant $c>0$ depending only on $w(\cdot)$ such that
		\[
		\inf_{\hat{T}}\sup_{(\bs{\t}, \bs{\S}) \in \T(s, d, \rho)} \esp\, w\big(\s^{-1}( \phi(s, \rho^2))^{-1/2}|\hat{T} - \|\bs{\t}\|_2| \big)  \ge c,
		\]
		where $\inf_{\hat{T}}$ denotes the infimum over all estimators. 
	\end{theorem}
	Let $\tilde{N}$ be an estimator defined in the same way as $\hat{N}$ with the only difference that $\|\bs{\S}\|_F$ is replaced by its upper bound $\rho$. The following theorem shows that $\tilde{N}$ is minimax optimal over the set $\T(s, d, \rho)$ with respect to the mean square risk.
	
	\begin{theorem}\label{thm:knownlb}
		We have
		\begin{equation}\label{eq1:thm:knownlb}
			\sup_{(\bs{\t}, \bs{\S}) \in \T(s, d, \rho)} \esp[(\tilde{N} - \|\bs{\t}\|_2)^2] \lesssim \s^2 \phi(s, \rho^2),  
		\end{equation}
		and 	
		\begin{equation}\label{eq2:thm:knownlb}
			\inf_{\hat{T}}\sup_{(\bs{\t}, \bs{\S}) \in \T(s, d, \rho)} \frac{\esp[(\hat{T} - \|\bs{\t}\|_2)^2]}{\phi(s, \rho^2)} \asymp \s^2.
		\end{equation}
	\end{theorem}
	We omit the proof of \eqref{eq1:thm:knownlb}  as it follows exactly the same lines as the proof of \eqref{thmeq:hatn} with $\|\bs{\S}\|_F$ replaced by $\rho$. The minimax property \eqref{eq2:thm:knownlb} follows by combining \eqref{eq1:thm:knownlb} and  
	Theorem \ref{thm:lower-known-sigma-rho} with $w(u)=u^2$.
	
	
	
	\section{Estimating the noise level}
	\label{sec:sigma}
	
		In what follows, we consider the model with unknown $\sigma$ and we impose 
		the following additional assumption on the diagonal elements of the matrix $\bs{\S}$.
	%
	\begin{condition}\label{as:var1}
		There exists a constant $c_* > 0$ such that the diagonal elements $\s_{i}^2$ of $\bs{\S}$ satisfy the condition $\min_{1\le i\le d}\s_{i}^2 \ge c_*$.
	\end{condition}
	Thus, the minimum of $\sigma_i^2$'s is controlled by a positive constant that does not depend on $d$. It implies that $\|{\bs{\S}}\|_F\gtrsim \sqrt{d}$. On the other hand, in this section, we do not need Assumption~\ref{as:var0}.
	
	Under Assumption~\ref{as:var1}, it is convenient to work with the normalized observations $\tilde Y_i= \tilde \t_i +\s \tilde{\varepsilon}_i$ for $i\in [d]$, where $\tilde \t_i=\t_i/\s_i$, and $\tilde{\varepsilon}_i=\varepsilon_i/\s_i$ is a standard Gaussian random variable. Thus, we replace the covariance matrix ${\bs{\S}}$ by the correlation matrix $\tilde{\bs{\S}}$ with entries
	\[
	\tilde{\bs{\S}}_{ij}=\bs{\S}_{ij}/(\s_i\s_j),\quad i,j\in[d].
	\]
	In this section, we propose estimators for the unknown noise level $\s^2$ that will be further used to construct the estimators of functionals adaptive to the noise level. We first provide an estimator~$\hat\s_S^2$, which is rate-optimal in the sparse regime $s\le \|\tilde{\bs{\S}}\|_F$.
	Next, we construct a refined estimator~$\hat \s_D^2$ that attains the optimal rate for all $s\lesssim d$.

	\subsection{Estimating $\s^2$ in the sparse regime}
	The first estimator $\hat\s_S^2$, designed for the sparse regime, is inspired by the techniques of robust estimation~\cite{HR09}.
	Note that all observations $\tilde Y_i$ with $i\not\in\cS$ are distributed as $\mathcal{N}(0,\s^2)$, and thus estimating $\s^2$ can be equivalently viewed as a robust estimation problem, where the observations with non-zero $ \tilde \t_i$'s play the role of outliers.
	In particular,
	let $\hat F_{Y^2}$ and $\hat F_{\varepsilon^2}$  be the empirical cumulative distribution functions of $\{\tilde Y_i^2: i\in[d]\}$ and $\{\tilde \varepsilon_i^2:i\in[d]  \}$, respectively.
	Then,
	\begin{equation}
		\label{eq:FY-Feps}
		\sup_{t> 0}|\hat F_{Y^2}(t) - \hat F_{\varepsilon^2}(t/\sigma^2)|\le \frac{\|\bs{\t}\|_0}{d}.
	\end{equation}
		Considering here some fixed $t>0$, approximating $\hat F_{\varepsilon^2}$ by the cumulative distribution function of the chi-squared distribution with one degree of freedom $\chi^2_1$ (that we denote by $F$) and solving $\hat F_{Y^2}(t)\approx F(t/\sigma^2)$ leads to the following estimator of $\s^2$: 
	\begin{equation}
		\label{eq:hat-sigma-t-sparse}
		\hat\s_S^2(t) := \frac{t}{F^{-1}(\hat F_{Y^2}(t))}.
	\end{equation}
	A relevant choice of $t$ in \eqref{eq:hat-sigma-t-sparse} is to be determined. 
		Notice that for any deterministic $t>0$ this estimator is not well-defined since the denominator in \eqref{eq:hat-sigma-t-sparse} can be 0 with positive probability. Indeed, for any $t>0$ there exist $\bs{\t} \in B_0(s, d)$ such that $\prob[\hat F_{Y^2}(t)=0]>0$. Thus, we will only define $\hat\s_S^2(t)$ by \eqref{eq:hat-sigma-t-sparse} if the denominator in \eqref{eq:hat-sigma-t-sparse} is non-zero. {\rev Otherwise, we set $\hat\s_S^2(t) = \sum_{i = 1}^d Y_i^2 / d$; we note that the conclusion remains valid for any other choice of this value.}
	The next lemma provides a bound on the estimation error of $\hat\s_S^2(t)$ for any fixed $t>0$.
	\begin{lemma}
		\label{lmm:sigma-t-prob}
		There exists an absolute constant $C>0$ such that 
			for any $\bs{\t} \in B_0(s, d)$ and any $\bs{\S}$ satisfying Assumption~\ref{as:var1}
		the following holds.\\
		(i) For any $t> 0$ and $0\le u \le  1/2$ we have
		\[
		\prob\qth{  \abs{\frac{\hat{\s}_S^2(t)}{\sigma^2}-1}\ge u }
		\le
		\frac{C}{u^2}\pth{\frac{s^2}{d^2}+\frac{ \|\tilde{\bs{\S}}\|_F^2 (  (\frac{t}{\s^2})^2 + 1)e^{-\frac{t}{3\sigma^2}} }{d^2}}
		\frac{e^{2t/\sigma^2}}{t/\sigma^2}.
		\]
		(ii) If, in addition, $|t / \s^2 - 1| \le \frac{2}{3}$ and $s / d \le 0.01$ then
		\[
		\prob\qth{  \abs{\frac{\hat{\s}_S^2(t)}{\sigma^2}-1}\ge \frac{1}{2} }
		\le
		C \frac{ \|\tilde{\bs{\S}}\|_F^2}{d^2}.
		\]
	\end{lemma}
	Lemma~\ref{lmm:sigma-t-prob} suggests that the estimator $\hat\s_S^2(t)$ achieves  a reasonable accuracy for $t\asymp \sigma^2$. However, such choice of $t$ is not realizable since $\sigma$ is unknown. Moreover, for any deterministic $t>0$ the expected error of  $\hat\s_S^2(t)$  
		is not bounded uniformly over $\bt$ since $\hat\s_S^2(t)$ can be $+\infty$ with some positive probability as pointed out above.
	Therefore, we turn to an alternative solution that consists in choosing a random $t=\hat t$ based on the sample. Note that the standard sample splitting strategy,
	under which $\hat t$ can be independently chosen, is not a viable option since we have a model with dependent  observations.
	We avoid sample splitting and propose a data-driven threshold $\hat t$ determined by scanning $\hat F_{Y^2}$ over a dyadic grid, where $\hat F_{Y^2}$ is based on the same set of observations as $\hat{\s}_S^2(\cdot)$ \footnote{
		The choice of constant 0.5 in the definition of $\hat t$ is not crucial. Replacing 0.5 by any other constant within $(0,1)$ only changes the resulting rates by a constant factor. }:
	\begin{equation}
		\label{eq:def-hat-t}
		\hat t
		:= \min \sth{t=2^\ell : \ell\in \ZZ, \hat F_{Y^2}(t) \ge 0.5
		}.
	\end{equation}
	We consider now the following estimator of $\s^2$:
	\[
	\hat\s_S^2 := \hat\s_S^2(\hat t).
	\]
	The next theorem gives a bound on the estimation error of $\hat\s_S^2$.
	
	\begin{theorem}
		\label{thm:hat-sigma-sparse}
		There exist absolute constants $C>0$ and $c>0$ such that for any $s\le c d$, 
			any $\bs{\t} \in B_0(s, d)$ and any $\bs{\S}$ satisfying Assumption~\ref{as:var1}
		we have
		\[
		\esp\left[\frac{|\hat{\s}_S^2 - \s^2|}{\s^2}\right] \le C \left(\frac{s}{d} + \frac{\|\tilde{\bs{\S}}\|_F}{d}\right).
		\]
	\end{theorem}
	
	
	It can be seen from the proof of Theorem~\ref{thm:hat-sigma-sparse} that the threshold \eqref{eq:def-hat-t} obtained by a search on a dyadic grid satisfies $\hat t \asymp \s^2$ with high probability. An alternative approach could be to use the sample median $\hat t = \min \{t: \hat F_{Y^2}(t) \ge 0.5\}$, which is also of the order of $\s^2$ with high probability. 
		The main limitation of this approach is that, under the dependent data, the resulting estimator $\hat{\s}^2(\hat{t})$ becomes difficult to analyze.
	On the other hand, as we prove below, the value $\hat{t}$ defined in \eqref{eq:def-hat-t} belongs with high probability to a \emph{finite set}
	\[
	\calT
	:= \sth{ t=2^{\ell}: \ell\in \ZZ, \frac{\sigma^2}{3}\le t \le \frac{3}{2}\sigma^2  }.
	\]
	This allows us to control $\esp\left[\frac{|\hat{\s}_S^2 - \s^2|}{\s^2}\right]$ by analyzing the convergence of $\hat{F}_{Y^2}(t)$ for every fixed $t \in \calT$. In other words, with our dyadic grid search we restrict the ``degrees of freedom'' in choosing $\hat{t}$, which makes the resulting estimator more robust against data dependence.
	{\rev In the next section, we provide another estimator of $\s^2$ that outperforms $\hat{\s}_S^2$ in the dense regime.}


	\subsection{{\rev Estimating $\s^2$ in the dense regime}}
		Note that the estimator \eqref{eq:hat-sigma-t-sparse} can be viewed as a method of moments estimator based on the empirical moments of indicator functions of $\tilde Y_i^2$'s.
	We can develop a similar argument with the empirical moments of another family of bounded functions, namely, the cosine functions of $\tilde Y_i$'s. 
	We consider the following empirical moments:
	\[
	\hat\varphi_d(u)
	:= \frac{1}{d}\sum_{i=1}^d \cos(u \tilde Y_i ), \quad u>0.
	\]
		Analogously to \eqref{eq:FY-Feps}, when $s$ is small compared to $d$ we can approximate $|\hat\varphi_d(u)|$ by the term $|\frac{1}{d}\sum_{i=1}^d \cos(u \sigma\tilde \varepsilon_i )|$, which in turn is approximated by the value of the normal characteristic function  $\varphi(u)=\exp(-u^2 \sigma^2/2)$. Solving $\exp(-u^2 \sigma^2/2) \approx |\hat\varphi_d(u)|$ suggests  $-\frac{2}{u^2}\log | \hat\varphi_d(u)|$ as an estimator of $\sigma^2$, under some choice of $u$.
	
	We set $\lambda =
	\frac{1\vee \log(s/\|\tilde{\bs{\S}}\|_F)}{6}$ and consider the estimator
	\[
	\tilde \s^2_D :=
	- \frac{2 \hat t }{\lambda}\log \Big| \hat\varphi_d \Big(\sqrt{\lambda/\hat t}\Big) \Big|,
	\]
	where $\hat t$ is defined in \eqref{eq:def-hat-t}.
	Our final estimator of the noise level is $\hat \s_D^2 := \min\{\tilde \s^2_D, 2\hat \s_S^2\}$.

	\begin{theorem}
		\label{thm:hat-sigma-dense}
		There exist absolute constants $C>0$ and $c>0$ such that for any $s\le c d$, 
			any $\bs{\t} \in B_0(s, d)$ and any $\bs{\S}$ satisfying Assumption~\ref{as:var1}
		we have
		\[
		\esp\left[\frac{|\hat\s_D^2 - \s^2|}{\s^2}\right] \le C \left(\frac{s}{d (1 \vee \log({s}/{\|\tilde{\bs{\S}}\|_F}) )} + \frac{\|\tilde{\bs{\S}}\|_F}{d}\right).
		\]
	\end{theorem}
	
	If we specify the difference between sparse and dense regimes, the bound of Theorem~\ref{thm:hat-sigma-dense} exhibits the rate  
	\[
	\tilde \psi (s, \bs{\S}) := \left\{
	\begin{array}{lc}
		\frac{\|\tilde{\bs{\S}}\|_F}{d} & \textrm{if}\; s \le \|\tilde{\bs{\S}}\|_F, \\
		\frac{s}{d(1 \vee \log(s / \|\tilde{\bs{\S}}\|_F))} & \textrm{otherwise.}
	\end{array}
	\right.
	\]
	

	Theorem~\ref{thm:hat-sigma-dense} provides a generalization of the upper bound proved in~\cite[Proposition 7]{comminges2021adaptive} for the case of i.i.d.\ noise. 
	Compared to the techniques used in~\cite{comminges2021adaptive}, our correlated noise setting brings fundamental challenges.
		The analysis in~\cite{comminges2021adaptive} relies on exponential concentration inequalities for the uniform deviations that we are short of as the observations are dependent and the spectral norm of $\bs{\S}$ can be as large as $O(\sqrt{d})$.
		The sample splitting technique for obtaining multiple independent estimates used in~\cite{comminges2021adaptive} also breaks down when the observations are correlated.
		It turns out that for our estimator, that does not use any sample splitting, second-order moment inequalities suffice to achieve the desired rate. 
			Note that we obtain an improvement upon Proposition 7 in~\cite{comminges2021adaptive} even in the case of i.i.d. observations since our estimator does not need sample splitting. 

		\section{Estimating the $\ell_2$-norm when the noise level is unknown}
		\label{sec:unknown-sigma}

		In this section, we use the estimators of $\s^2$ proposed in Section~\ref{sec:sigma} in order to construct estimators of the $\ell_2$-norm $\| \bs{\t} \|_2$ when the noise level $\sigma$ is unknown. 
		
		Assume first that the Frobenius norm of the covariance matrix $\|\bs{\S}\|_F$ is known. We start by considering an estimator of $\|\bs{\t}\|_2$ defined as follows: 
		$$\hat{N}^* := \sqrt{|\hat{Q}^*|},$$ where
		\begin{equation}
			\label{eq:Qstar}
			\hat{Q}^* :=
			\begin{cases}
				\sum_{i=1}^d Y_i^2 \one\{|Y_i| > \hat{\s}_S \s_{i} \tau \} - \hat{\s}_S^2 \alpha_s \sum_{i=1}^d \s_{i}^2, & \textrm{if}\; s \le \|\bs{\S}\|_F, \\
				\sum_{i=1}^d (Y_i^2 - \hat{\s}_D^2 \s_{i}^2), & \textrm{otherwise}.
			\end{cases}
		\end{equation}
		Here, $\alpha_s := {\mathbf E}\left[Z^2 \one \{|Z| \ge \tau \}\right]$ with $Z\sim \mathcal{N}(0,1)$, and $\tau = 3\sqrt{\log(1 + \|\bs{\S}\|_F^2 / s^2)}$.
		Set
		\[
		\psi^*(s, \bs{\S}): = \phi^*(s, \|\bs{\S}\|_F^2) = \left\{
		\begin{array}{lc}
			s \log(1 + \|\bs{\S}\|_F^2 / s^2) & \textrm{if}\; s \le \|\bs{\S}\|_F, \\
			\frac{s}{1 \vee \log(s / \|\bs{\S}\|_F)} & \textrm{otherwise.}
		\end{array}
		\right.
		\]
		Our aim in this section is to show that $\s^2\psi^*(s, \bs{\S})$ is the minimax optimal rate for the squared error of estimating $\|\bs{\t}\|_2$ on $\Theta(s,d)$ if Assumptions~\ref{as:var0} and~\ref{as:var1} hold, and that the estimator $\hat N^*$ or its modifications attain this rate.
		The following theorem provides bounds on the accuracy of the estimator $\hat{N}^*$.
		\begin{theorem}\label{thm:unknownvar1}
			There exists an absolute constant $\alpha \in (0, 1)$ such that, for any 
				$\bs{\t} \in B_0(s, d)$ with $s$ satisfying that $s \le \alpha d$ and any  $\bs{\S}$ satisfying Assumptions~\ref{as:var0} and~\ref{as:var1},
			the following holds.
			\begin{enumerate}
				\item[(i)] In the regime $s > \|\bs{\S}\|_F$, there exists an absolute constant $c_1 > 0$ such that
				\[
				\esp[(\hat{N}^* - \|\bs{\t}\|_2)^2] \le c_1 \s^2 \psi^*(s,\bs{\S}).
				\]
				\item[(ii)] In the regime $s \le \|\bs{\S}\|_F$, there exists an absolute constant $c_2>0$ such that, for any $\eta \in (c_2 \|\bs{\S}\|_F / d, 1)$, the following inequality holds with some constant $c_\eta > 0$ depending only on~$\eta$:
				\begin{equation}
					\label{eq:unknownvar1-sparse1}
					\prob\left((\hat{N}^* - \|\bs{\t}\|_2)^2 \le c_\eta \s^2 \psi^*(s, \bs{\S})\right) \ge 1 - \eta.
				\end{equation}
				\item[(iii)] If, in addition, $\bs{\S}$ is a positive definite matrix there exists an absolute constant $c_3 > 0$ such that
					\begin{equation}
						\label{eq:unknownvar1-sparse2}
						\begin{aligned}
							\esp[(\hat{N}^* - \|\bs{\t}\|_2)^2] \le & \; c_3 \left(\sqrt{\lambda_{\max}(\bs{\S}^{-1})} + 1\right) \s^2 \psi^*(s, \bs{\S}) \\
							&\qquad \qquad +
							c_3 \s^2 \frac{\|\bs{\S}\|_F^2}{d} \log(1 + \|\bs{\S}\|_F^2 / s^2).
						\end{aligned}
					\end{equation}
			\end{enumerate}
		\end{theorem}
		
		Theorem~\ref{thm:unknownvar1} grants an expected squared risk bound with the desired rate $\s^2\psi^*(s,\bs{\S})$ in the dense regime $s > \|\bs{\S}\|_F$. However, in the sparse regime $s \le \|\bs{\S}\|_F$ Theorem~\ref{thm:unknownvar1} implies such a conclusion only under additional conditions. The bound \eqref{eq:unknownvar1-sparse1} provides a bound in probability under an additional ($\bs{\S}$-dependent) condition on $\eta$.  Part (iii) of Theorem~\ref{thm:unknownvar1} implies that,  under the additional assumption that $\bs{\S}$ is positive definite with minimum eigenvalue bounded below by some positive constant and $\|\bs{\S}\|_F^2\le sd$, the proposed estimator achieves the rate $\s^2\psi^*(s,\bs{\S})$ for the expected squared risk also in the sparse regime. 
		
			The bound~\eqref{eq:unknownvar1-sparse1} is valid only under the condition that $\eta$ is bigger than $c_2 \|\bs{\S}\|_F / d$, which is not negligible when $\|\bs{\S}\|_F\asymp d$. We now show that one can construct an estimator dependent of $\eta$ that is free of this constraint. Namely, for any given $\eta\in (0, 1)$, we propose an estimator depending on $\eta$ that  achieves~\eqref{eq:unknownvar1-sparse1}. Fix $\eta \in (0, 1)$ and let $q_\eta$ be the $\eta$-th quantile of a chi-square random variable with one degree of freedom. Define
			\[
			\hat{\s}_\eta^2 := \frac{\hat{M}}{q_{1 - \eta / 20}}, \quad\textrm{where}\quad \hat{M} := \min \{t: \hat{F}_Y^2(t) \ge 0.5\}.
			\]
			We now modify $\hat{Q}^*$ in the sparse regime $s \le \|\bs{\S}\|_F$ as follows:
			\[
			\hat{Q}^*_\eta := \sum_{i=1}^d Y_i^2 \one\{|Y_i| > \max\{\hat{\s}_S, \hat{\s}_\eta\} \s_{i} \tau_\eta \} - \hat{\s}_S^2 \alpha_{s, \eta} \sum_{i=1}^d \s_{i}^2,
			\]
			where 
			\[
			\tau_\eta = 3 \sqrt{\frac{q_{1 - \eta / 20}}{q_{\eta / 20}}\log(1 + \|\bs{\S}\|_F^2 / s^2)} \quad\textrm{and}\quad \alpha_{s, \eta} = {\mathbf E}[Z^2 \one\{|Z| \ge \tau_\eta\}].
			\]
			The corresponding estimator of the $\ell_2$-norm $\hat{N}^*_\eta := \sqrt{|\hat{Q}^*_\eta|}$ achieves the following bound in probability in the sparse regime. 
			\begin{theorem}\label{thm:unknownvar2}
				There exists a constant $\alpha \in (0, 1)$ such that for any $\eta \in (0, 1)$, any $\bs{\t} \in B_0(s, d)$ with $s$ satisfying that $s \le \min\{\alpha d, \|\bs{\S}\|_F\}$, and any  $\bs{\S}$ satisfying Assumptions~\ref{as:var0} and~\ref{as:var1} we have 
				\[
				\prob\left((\hat{N}^*_\eta - \|\bs{\t}\|_2)^2 \le c_\eta \s^2\psi^*(s, \bs{\S})\right) \ge 1 - \eta,
				\]
				where $c_\eta > 0$ is a constant depending only on $\eta$.
			\end{theorem}

		Next, we consider the setting where the Frobenius norm $\|\bs{\S}\|_F$ is not available and we only know an upper bound $\rho$ on $\|\bs{\S}\|_F$. {\rev In this setting, we show that estimators that are based on slight modifications of $\hat{N}^\star$ and $\hat{N}^\star_\eta$ can achieve minimax optimality on the parameter spaces~$\T(s,d,\rho)$, $\rho>0$, where the restriction  $\|\bs{\S}\|_F\le \rho$ is active.}
		Let $\tilde{N}^*$ and $\tilde{N}^*_\eta$ be the estimators defined in the same way as $\hat{N}^*$ and $\hat{N}^*_\eta$, respectively, with the only difference that $\|\bs{\S}\|_F$ is replaced by its upper bound~$\rho$. Let $\phi^*(s, \cdot)$ be defined in \eqref{rate-comminges-unknown sigma}. Then $\phi^*(s, \rho^2)$ follows the same definition as $\psi^*(s, \bs{\S})$ except that we replace $\|\bs{\S}\|_F$ by $\rho$.

		\begin{theorem}\label{thm:unknownlb}
				Let $w:[0,\infty)\to[0,\infty)$ be a monotone non-decreasing function such that $w(0)=0$ and $w\not\equiv 0$. Assume that there exists a constant $\alpha \in (0, 1)$ such that $s \le \alpha d$. Then there exists a constant $c>0$ depending only on $w(\cdot)$ such that
				\[
				\inf_{\hat{T}}\sup_{(\bs{\t}, \bs{\S}) \in {\T}(s, d, \rho)} \sup_{\s > 0} \esp\, w\big(\s^{-1}( \phi^*(s, \rho^2))^{-1/2}|\hat{T} - \|\bs{\t}\|_2| \big)  \ge c.
				\]
			Moreover,
			\begin{enumerate}
				\item [(i)] In the regime $s > \rho$,
				\[
				\sup_{(\bs{\t}, \bs{\S}) \in \T(s, d, \rho)} \sup_{\s > 0} \frac{\esp[(\tilde{N}^* - \|\bs{\t}\|_2)^2]}{\s^2} \lesssim \phi^*(s, \rho^2).
				\]
				\item [(ii)] In the regime $s \le \rho$, for any $\eta \in (0,1)$ there exists a constant $c_\eta>0$ depending only on $\eta$ such that
				\[
				\sup_{(\bs{\t}, \bs{\S}) \in \T(s, d, \rho)} \sup_{\s > 0} \prob\left(({\tilde{N}^*_\eta} - \|\bs{\t}\|_2)^2 \ge c_\eta \s^2 \phi^*(s, \rho^2)\right) \le \eta.
				\]
			\end{enumerate}
		\end{theorem}
		
		\section{Optimal testing for signal detection}
		\label{sec:tests}
		
		In this section, we provide some corollaries of the above results for the problem of testing for signal detection. We now introduce the notation $\mathcal{M}_\rho:=\{\bs{\S}: \bs{\S}_{ii}=1, i\in[d], \text{ and } \|\bs{\S}\|_F \le \rho \}$, where $\rho>0$; then problem of testing hypotheses can be defined as follows:
		\begin{eqnarray*}
			H_0: \bs{\t}=0, \bs{\S}\in \mathcal{M}_\rho \quad \text{against} \quad H_1: \, (\bs{\t}, \bs{\S})\in \Theta(s,d, \rho) : \|\bt\|_2 \ge r,
		\end{eqnarray*}
		where the parameter $r>0$ characterizes a separation between the two hypotheses. When $r$ is too small there is no hope to construct a good test for this problem. The main question is to find the smallest, in a minimax sense, separation radius $r$, under which one can obtain a valid test and to provide such a test. 
		
		For any test $\Psi$, that is, any measurable function of $\bY$ with values in $\{0,1\}$, we define its risk as the sum of
		type I and type II error probabilities:
		$$
		{\mathcal R}(\Psi,s,d,\rho,r)= \sup_{\bs{\S}\in \mathcal{M}_\rho} \mathbf{P}_{0,\bs{\S}}(\Psi=1)+ 
		\sup_{(\bs{\t}, \bs{\S})\in \Theta(s,d,\rho) : \|\bt\|_2 \ge r}\mathbf{P}_{\bs{\t},\bs{\S}}(\Psi=0).
		$$
		The minimax risk is the smallest risk among all tests $\Psi$, that is, 
		$$
		{\mathcal R}^*(s,d,\rho,r) = \inf_{\Psi} {\mathcal R}(\Psi,s,d,\rho,r).
		$$
		Assume that we are given a tolerance level $\eta\in (0,1)$ for the minimax risk. 
		The {\it minimax rate of testing} $r^*_{\eta}(s,d,\rho)$ is defined as
		$$
		r^*_{\eta}(s,d,\rho) = \inf\left\{r>0: \,  {\mathcal R}^*(s,d,\rho,r)\le \eta  \right\}.
		$$
		A natural approach to testing $H_0$ against $H_1$ is to reject $H_0$ when some estimator of the $\ell_2$-norm $\|\bt\|_2$ exceeds a suitably chosen threshold. We use this approach with the estimators of the $\ell_2$-norm proposed above and the thresholds derived from the optimal rates of estimation that we have obtained. Thus, consider the test
		$$
		\Psi^*= \one\big(\hat N > \gamma\s \sqrt{\phi(s,\rho^2)}\big),
		$$
		where $\gamma>0$ is a constant. Introduce the separation radius
		$$\bar r_{\gamma} = 2\gamma \s \sqrt{\phi(s,\rho^2)}.$$
		Using the upper bound \eqref{eq1:thm:knownlb} and Chebyshev's inequality we can control the risk of $\Psi^*$ for the separation radius $\bar r_{\gamma}$  as follows:
		\begin{align}
			{\mathcal R}(\Psi^*,s,d,\rho,\bar r_{\gamma}) &\le  \sup_{\bs{\S}\in \mathcal{M}_\rho }\mathbf{P}_{0,\bs{\S}}(\hat N > \gamma \s\sqrt{\phi(s,\rho^2)})
			\\
			&\qquad +
			\sup_{(\bs{\t}, \bs{\S})\in \Theta(s,d,\rho)}\mathbf{P}_{\bs{\t},\bs{\S}}(\hat N - \|\bt\|_2 \le - \gamma \s\sqrt{\phi(s,\rho^2)})
			\\
			&\le
			2 \sup_{(\bs{\t}, \bs{\S})\in \Theta(s,d,\rho)} \frac{\esp[(\hat N - \|\bt\|_2)^2]}{\gamma^2 \s^2\phi(s,\rho^2)} \le C^* \gamma^{-2},
		\end{align}
		where $C^*>0$ is an absolute constant. Choosing $\gamma=\sqrt{\eta/C^*}$ implies that the last expression does not exceed $\eta$, so that 
		\begin{equation}\label{eq:ub-test}
			r^*_{\eta}(s,d,\rho)\le 2\sqrt{\eta/C^*}\s\sqrt{\phi(s,\rho^2)}.  
		\end{equation}
		On the other hand, 
		a matching lower bound on $r^*_{\eta}(s,d,\rho)$ follows from the next proposition.
		\begin{proposition}\label{prop: lower-testing}
			There exists an absolute constant $c_4 > 0 $ such that given any $\gamma > 0$, we have
			\[
			{\mathcal R}^*(s,d,\rho,\bar r_{\gamma}) \ge \exp(1 - \exp(c_4 \gamma^2)) / 2.
			\]
		\end{proposition}
		\noindent
		For any  $\eta\in (0,1/2)$,  the equation $\exp(1 - \exp(c_4 \gamma^2)) / 2 =\eta$ has the solution $$\gamma_\eta=\sqrt{(c_4)^{-1}\log\log\big(\frac{e}{2\eta}\big)}.$$ 
		Thus, for $\eta\in (0,1/2)$, we get ${\mathcal R}^*(s,d,\rho,\gamma_\eta \s \sqrt{\phi(s,\rho^2)})\ge \eta$, 
			which implies the bound 
			\begin{equation}\label{eq:lb-test}
				r^*_{\eta}(s,d,\rho)\ge \gamma_\eta\s\sqrt{\phi(s,\rho^2)}, \quad \forall  \eta\in (0,1/2). 
			\end{equation}
			Combining \eqref{eq:ub-test} and \eqref{eq:lb-test} we obtain the following corollary.
			\begin{corollary}\label{corollary:test}
				For any  $\eta\in (0,1/2)$, there exist constants $c_\eta, C_\eta$ depending only on $\eta$ such that the minimax rate of testing satisfies
				$$
				c_\eta \s\sqrt{\phi(s,\rho^2)} \le r^*_{\eta}(s,d,\rho) \le C_\eta \s\sqrt{\phi(s,\rho^2)}.
				$$
			\end{corollary}
			%
			%
			
			{\rev 
				The results of this section deal with minimax optimality on a class of matrices, while the prior work on sparse testing discussed in the Introduction was analyzing minimax optimality for a fixed  matrix $\bs{\S}$. 
				In those papers,  $H_0$ was a simple hypothesis (with a given matrix $\bs{\S}$, which was assumed diagonal, except for \cite{hall2009jin,kotekal2021minimax,Gao2021change-point}), while here we deal with minimax optimality over classes of matrices $\bs{\S}$, and in our case $H_0$ is a composite hypothesis. Of course, for specific fixed matrices in the class one gets, in general, better rates than those achieved uniformly over a class. This can be seen in the setting with equicorrelation matrix detailed in \citet{kotekal2021minimax}, and in \citet{hall2009jin} for a different instance of sparse testing problem in asymptotics as $d\to\infty$.  
				
			}

			In \cite[Theorem 7]{Gao2021change-point} the authors provide the minimax rate for a testing problem in a change-point framework, which implies the minimax rate of the order $\sqrt{\|\bs{\S}\|_F}$ for testing in model \eqref{eq:model} under the no sparsity scenario $s=d$ with fixed positive definite matrix $\bs{\S}$ and $\s=1$.  Note that it agrees with our conclusions for this particular setting. Indeed, when $\bs{\S}$ is fixed with entries independent of $d$ we have $\|\bs{\S}\|_F\lesssim d$, so that for $s=d$ our Theorem \ref{thm:knownvar} gives the rate $\psi(s, \bs{\S})=\|\bs{\S}\|_F$ for the squared loss leading to the $\sqrt{\|\bs{\S}\|_F}$ rate of testing for this particular case.
			
		{\rev
			\begin{remark}
				An interesting phenomenon described in \citet{hall2009jin} and further noticed in other context by \citet{kotekal2021minimax}
				consists in the fact that, for some specific families of matrices $\bs{\S}$, the  minimax rate of testing can be better than in the case of independent noise, and the hardest subproblem corresponds to the identity covariance matrix.  
				In contrast, in our setting this is not the case. Indeed, inspection of the proof shows that our lower bound is obtained on correlated data that correspond to block-diagonal matrices and not to the identity covariance matrix. 
				The independent case appears in the proof only as a technical element -- after a suitable reduction we apply the lower bounds for the independent case from \citet{CCT2017}. 
				Our contribution is a reduction argument that leads to a contraction of  
				the original problem to testing two composite hypotheses on a smaller scale. Importantly, the reduction 
				step itself does not rely on the hypotheses being composite.  
				The intuition is that, with a stronger correlation on $\bs{\S}$, the effective sample size of the noise vector $\bs{\varepsilon}$ becomes smaller, with 
				the effective sample size being directly governed by the parameter $\rho$. 
				In our lower bound construction, this reduction is realized through a block-wise covariance structure, where the block size is also determined by $\rho$. 
				Consequently, the estimation error of the mean parameter within each block is amplified by the block size. 
				This construction provides a convenient and analytically tractable instance that captures the essential difficulty of the problem.   
			\end{remark}
		}
		
		\section{Proofs}\label{sec:proof}
		\subsection{Properties of correlated Gaussian variables}
		\label{sec:properties-gaussian}
		
		In this section, we collect some facts about correlated Gaussian variables that will be useful in the proofs. 
		We show that several functions of Gaussian random variables scale quadratically with the correlation coefficients, which will be crucial to obtain the optimal dependence of the rates of convergence on the Frobenius norm of the covariance matrix.

		\begin{lemma}\leavevmode
			\label{lmm:correlated-gaussian}
			There exists an absolute constant $C>0$ such that the following statements hold for two  standard Gaussian random variables $\zeta$ and $\eta$ that are jointly normally distributed with correlation coefficient $\nu$.
			\begin{enumerate}[label = (\roman*)]
				\item \label{lmm:square} $ \cov[\zeta^2, \eta^2] = 2\nu^2. $
				\item \label{lmm:truncate-moments} For any $\tau\ge 1$, let $\beta=
				{\mathbf E}[Z^2 \big| |Z|>\tau]$ where $Z\sim \calN(0,1)$. Then
				\[
				\cov[(\zeta^2-\beta)\one(|\zeta|>\tau), (\eta^2-\beta)\one(|\eta|>\tau)]
				\le C \nu^2 \tau^4 \exp(-\tau^2/2).
				\]
				\item \label{lmm:cdf} For any $\tau \ge 0$,
				\[
				| \cov[\one(\zeta^2 \le \tau), \one(\eta^2 \le \tau)] |
				\leq  C \nu^2 (\tau^2+1) e^{-\tau/3}.
				\]
				\item \label{lmm:cosine} For any $t, \mu_1,\mu_2 \in \RR$, $\s_1,\s_2>0$ we have
				\begin{align}
					& | \cov[\cos(t (\mu_1+\s_1\zeta)), \cos(t (\mu_2+\s_2\eta))] | \\
					&\qquad \le C \big( |\nu^2\cos(t \mu_1) \cos(t \mu_2)|  + |\nu\sin(t \mu_1) \sin(t \mu_2)|\big).
				\end{align}
				\[
				\]
			\end{enumerate}
		\end{lemma}

		\begin{proof}
			Given $\zeta=u$, the conditional distribution of $\eta$ is $ \calN(\nu u,1-\nu^2)$. Therefore, $\eta$ has the same distribution as $\nu \zeta + \sqrt{1 - \nu^2} Z$ where $Z$ is a standard Gaussian random variable independent of $\zeta$.
			\begin{enumerate}[label = (\roman*)]
				\item Let $\eta = \nu \zeta + \sqrt{1 - \nu^2} Z$. We obtain
				\begin{align}
					& \cov[\zeta^2, \eta^2]  = {\mathbf E}[(\zeta^2 - 1) (\eta^2 - 1)] = {\mathbf E} [(\zeta^2 - 1) ((\nu \zeta + \sqrt{1 - \nu^2} Z)^2 - 1)] \\
					& = {\mathbf E}[(\zeta^2 - 1) (\nu^2 \zeta^2 + (1 - \nu^2) Z^2 + 2 \nu \sqrt{1 - \nu^2} \zeta Z - 1)] \\
					& = \nu^2 {\mathbf E}[(\zeta^2 - 1) \zeta^2] + 2 \nu \sqrt{1 - \nu^2} {\mathbf E}[(\zeta^2 - 1) \zeta Z ] + {\mathbf E}[(\zeta^2 - 1) ((1 - \nu^2)Z^2 - 1) ] \\
					& = 2\nu^2.
				\end{align}

				\item See~\cite[Lemma 5]{carpentier2019minimax}.

				\item We consider separately the cases $|\nu|\ge \frac{1}{5}$ and $|\nu|< \frac{1}{5}$. For $|\nu|\ge \frac{1}{5}$, using the Cauchy-Schwarz inequality we get
				\[
				| \cov[\one(\zeta^2 \le \tau), \one(\eta^2 \le \tau)] |
				\le \var[\one(\eta^2 \le \tau)]
				\le \mathbf{P}[\eta^2>\tau]
				\le C \nu^2 e^{-\tau/2},
				\]
				where $C>0$ is an absolute constant.
				
				Next we prove the result for $|\nu|< \frac{1}{5}$.
				Given $\zeta=u$, the conditional distribution of $\eta$ is $ \calN(\nu u,1-\nu^2)$.
				Then,
				\begin{align}
					&\cov[\one(\zeta^2 \le \tau), \one(\eta^2 \le \tau)]
					= \mathbf{P}[\zeta^2\le \tau, \eta^2\le \tau]-\mathbf{P}[\zeta^2\le \tau]\mathbf{P}[\eta^2\le \tau]\\
					=~&\int_{-\sqrt{\tau}}^{\sqrt{\tau}}\frac{1}{\sqrt{2\pi}}e^{-u^2/2}
					\left(
					\Phi\left(\frac{\sqrt{\tau}-\nu u}{\sqrt{1-\nu^2}}\right)
					-\Phi\left(\frac{-\sqrt{\tau}-\nu u}{\sqrt{1-\nu^2}}\right)
					-\Phi(\sqrt{\tau})+\Phi(-\sqrt{\tau})
					\right)du.\label{eq:cov-integral}
				\end{align}
				For the expression under the integral in \eqref{eq:cov-integral} we have
				\begin{align}
					&\left|
					\Phi\left(\frac{\sqrt{\tau}-\nu u}{\sqrt{1-\nu^2}}\right)
					-\Phi\left(\frac{-\sqrt{\tau}-\nu u}{\sqrt{1-\nu^2}}\right)
					-\Phi(\sqrt{\tau})+\Phi(-\sqrt{\tau})
					\right|\\
					= &\left|
					\Phi\left(\frac{\sqrt{\tau}-\nu u}{\sqrt{1-\nu^2}}\right) + \left(1
					-\Phi\left(\frac{-\sqrt{\tau}-\nu u}{\sqrt{1-\nu^2}}\right)\right)
					-\Phi(\sqrt{\tau}) - \left(1 - \Phi(-\sqrt{\tau})\right)
					\right|\\
					=~& \left|\Phi\left(\frac{\sqrt{\tau}+\nu u}{\sqrt{1-\nu^2}}\right)
					+\Phi\left(\frac{\sqrt{\tau}-\nu u}{\sqrt{1-\nu^2}}\right)
					-2\Phi\left(\frac{\sqrt{\tau}}{\sqrt{1-\nu^2}}\right) + 2\Phi\left(\frac{\sqrt{\tau}}{\sqrt{1-\nu^2}}\right)
					-2\Phi\left(\sqrt{\tau}\right)\right|\\
					\le~& \max_{x \in I_1}|\Phi''(x)|\left(\frac{\nu u}{\sqrt{1-\nu^2}}\right)^2
					+2\max_{x \in I_2}|\Phi'(x)|\left(\frac{\sqrt{\tau}}{\sqrt{1-\nu^2}}-\sqrt{\tau}\right),\label{eq:integrand-two-parts}
				\end{align}
				where $I_1$ is the interval with endpoints $\frac{\sqrt{\tau}-\nu u}{\sqrt{1-\nu^2}}$ and $\frac{\sqrt{\tau}+\nu u}{\sqrt{1-\nu^2}}$, $I_2=[\sqrt{\tau}, \frac{\sqrt{\tau}}{\sqrt{1-\nu^2}}]$, and we have used the triangle inequality and the relation between divided differences and derivatives (see, e.g., \cite[(2.1.4.3)]{SB2002}).
				Since $|\nu|\le \frac{1}{5}$ and $|u|\le \sqrt{\tau}$, we have $I_1\subseteq [\sqrt{\frac{2\tau}{3}},\sqrt{\frac{3\tau}{2}}]$.
				Note that $\Phi'(x)=\frac{1}{\sqrt{2\pi}}e^{-x^2/2}$ and $\Phi''(x)=-\frac{x}{\sqrt{2\pi}}e^{-x^2/2}$, and $\frac{1}{\sqrt{1-\nu^2}}-1=\frac{\nu^2}{\sqrt{1-\nu^2}+1-\nu^2}$.
				Therefore, the upper bound \eqref{eq:integrand-two-parts} yields
				\begin{align}
					& \left|
					\Phi\left(\frac{\sqrt{\tau}-\nu u}{\sqrt{1-\nu^2}}\right)
					-\Phi\left(\frac{-\sqrt{\tau}-\nu u}{\sqrt{1-\nu^2}}\right)
					-\Phi(\sqrt{\tau})+\Phi(-\sqrt{\tau})
					\right| \le C' \nu^2 (\tau^{3/2}+\tau^{1/2})e^{-\tau/3},
				\end{align}
				where $C'>0$ is an absolute constant. The conclusion follows from \eqref{eq:cov-integral}.

				\item  Taking the real part of the characteristic function we obtain that, for $X\sim \calN(\mu,\sigma^2)$,  
				\begin{equation}
					\label{eq:expect-cos}
					{\mathbf E}[\cos(t X)] = \cos(t\mu)\exp\pth{-\frac{1}{2}t^2\sigma^2}.
				\end{equation}
				Set $\delta = \nu \frac{2\sigma_1\sigma_2}{\sigma_1^2+\sigma_2^2}$, $\tilde\zeta = \mu_1+\s_1\zeta$, and $\tilde\eta = \mu_2+\s_2\eta$.
				Then, $\tilde\zeta+\tilde\eta\sim \calN(\mu_1+\mu_2, (\sigma_1^2+\sigma_2^2)(1+\delta))$ and $\tilde\zeta-\tilde\eta\sim \calN(\mu_1-\mu_2, (\sigma_1^2+\sigma_2^2)(1-\delta))$.
				Let $b=\frac{t^2(\sigma_1^2+\sigma_2^2)}{2}$. We get
				\begin{align}
					& \phantom{{}={}}
					\cov[\cos(t \tilde\zeta), \cos(t \tilde\eta)]
					={\mathbf E}[\cos(t \tilde\zeta)\cos(t \tilde\eta)] - {\mathbf E}[\cos(t \tilde\zeta)] {\mathbf E}[\cos(t \tilde\eta)]\\
					& = \frac{1}{2}{\mathbf E}[\cos t(\tilde\zeta+\tilde\eta)] + \frac{1}{2}{\mathbf E}[\cos t(\tilde\zeta-\tilde\eta)]
					- {\mathbf E}[\cos(t \tilde\zeta)] {\mathbf E}[\cos(t \tilde\eta)] \\
					& = \frac{1}{2}\cos(t\mu_1)\cos(t\mu_2)\pth{e^{-b (1+\delta) } + e^{-b(1-\delta)} - 2e^{-b}   } \\
                    & \phantom{{}={}} - \frac{1}{2}\sin(t\mu_1)\sin(t\mu_2)\pth{ e^{-b (1+\delta) } - e^{-b(1-\delta)} }.
					\end{align}
					It remains to show that
					\begin{align}
						\abs{e^{-b (1+\delta) } + e^{-b(1-\delta)} - 2e^{-b}}
						\lesssim \nu^2,
						\quad
						\abs{e^{-b (1+\delta) } - e^{-b(1-\delta)}}
						\lesssim |\nu|.
					\end{align}
					For $|\nu|\ge \frac{1}{2}$ these upper bounds hold trivially.
					For $|\nu|\le \frac{1}{2}$, using the fact that $|\delta|\le |\nu|\le \frac{1}{2}$ and applying the relation between divided differences and derivatives (see, e.g., \cite[(2.1.4.3)]{SB2002} we obtain
					\begin{align}
						\abs{e^{-b (1+\delta) } + e^{-b(1-\delta)} - 2e^{-b}}
						& \le \max_{\xi\in[\frac{1}{2},\frac{3}{2}] }  b^2 \delta^2 e^{-b \xi }  
						\lesssim e^{-b/4}\delta^2
						\le \nu^2,\\
						\abs{e^{-b (1+\delta) } - e^{-b(1-\delta)}}
						& \le \max_{\xi\in[\frac{1}{2},\frac{3}{2}] } 2b |\delta| e^{-b \xi } 
						\lesssim e^{-b/4}|\delta|
						\le |\nu|.
					\end{align}
					The conclusion follows.
					\qedhere
				\end{enumerate}
			\end{proof}
			As a corollary, we obtain the following proposition on the properties of various statistics of the correlated Gaussian random vector $\bs{\varepsilon}=(\varepsilon_1,\dots,\varepsilon_d)\sim \calN(0, \bs{\S})$.
			Recall that we denote  by $\s_1^2,\dots,\s_d^2$ the diagonal elements of $\bs{\S}=({\bs{\S}}_{ij})_{i,j}$ and by
			$\tilde{\bs{\S}}$ be the matrix of correlation coefficients with entries $\tilde{\bs{\S}}_{ij}={\bs{\S}}_{ij}/(\s_i\s_j)$.
			
			\begin{proposition}\leavevmode
				\label{prop:correlated-gaussian}
				Let $\bs{\varepsilon}\sim \calN(0, \bs{\S})$. There exists an absolute constant $C>0$ such that the following holds. 
				\begin{enumerate}[label = (\roman*)]
					\item \label{prop:square} $\var( \| \bs{\varepsilon} \|_2^2 )=2\| \bs{\S} \|_F^2$. \label{prop:var-chi}
					\item \label{prop:truncate-moments} Let $\tau\ge 1$ and $\beta=\mathbf{E}[Z^2 \big| |Z|>\tau]$, where $Z\sim \calN(0,1)$. Then, for any subset $\mathcal{U}\subseteq [d]$,
					\begin{align}
						& \var\qth{\sum_{i\in \mathcal{U}} (\varepsilon_i^2 - \s_i^2\beta) \one(|\varepsilon_i|>\s_i\tau) }  \le C \|\bs{\S}\|_F^2 \tau^4 \exp(-\tau^2/2), \\
						& \var\qth{\sum_{i\in \mathcal{U}} \varepsilon_i^2 \one(|\varepsilon_i|>\s_i\tau) }  \le C \|\bs{\S}\|_F^2 \tau^8 \exp(-\tau^2/3).
					\end{align}
					\item \label{prop:cdf} Let $\tau\ge 0$. Then
					\[
					\var\qth{\sum_{i\in \mathcal{U}} \one(\varepsilon_i^2\le \s_i^2\tau)} \le C \|\tilde{\bs{\S}}\|_F^2 (\tau^2+1)e^{-\tau/3},
					\]
					\item \label{prop:cosine} Let $t\in\RR$, and let $\mu_1,\dots,\mu_d\in\RR$ be such that $\sum_{i=1}^d\one(\mu_i\ne 0)\le s$. Then,
					\[
					\var\qth{\sum_{i=1}^d \cos(t(\mu_i+\varepsilon_i))}
					\le C\pth{\|\tilde{\bs{\S}}\|_F^2 + s \|\tilde{\bs{\S}}\|_F}.
					\]
				\end{enumerate}
			\end{proposition}

			\begin{proof}
				Items \ref{prop:square} and \ref{prop:cdf} follow directly from Lemma~\ref{lmm:correlated-gaussian}.
				
				\ref{prop:truncate-moments} Consider the standard Gaussian random variables $\tilde \varepsilon_i = \varepsilon_i / \s_i$, $i\in[d]$. The correlation coefficient between $\tilde \varepsilon_i$ and $\tilde \varepsilon_i$ is $\bs{\S}_{ij}/(\s_i \s_j)$ for $i,j\in[d]$. Thus, we have
				\begin{align}
					& \var\qth{\sum_{i\in \mathcal{U}} (\varepsilon_i^2 - \s_i^2\beta) \one(|\varepsilon_i|>\s_i\tau) } \\
					&=\sum_{i\in \mathcal{U}, j\in \mathcal{U} } \s_i^2\s_j^2\cov((\tilde\varepsilon_i^2 - \beta) \one(|\varepsilon_i|>\s_i\tau),(\tilde\varepsilon_j^2 - \beta) \one(|\varepsilon_j|>\s_j\tau))\\
					&\lesssim  \|\bs{\S}\|_F^2 \tau^4 \exp(-\tau^2/2),    \label{eq:truncate-moment-1}
				\end{align}
				where the last inequality follows from Lemma~\ref{lmm:correlated-gaussian}\ref{lmm:truncate-moments}.
				Next we prove the second inequality in \ref{prop:truncate-moments}. Note that
                \begin{align}
                    \var\qth{\sum_{i\in \mathcal{U}} \varepsilon_i^2 \one(|\varepsilon_i|>\s_i\tau) }
				\le & 2\var\qth{\sum_{i\in \mathcal{U}} (\varepsilon_i^2 - \s_i^2\beta) \one(|\varepsilon_i|>\s_i\tau) } \\
                & + 2 \var\qth{\sum_{i\in \mathcal{U}} \s_i^2\beta \one(|\varepsilon_i|>\s_i\tau) }.
                \end{align}
				Applying Lemma~\ref{lmm:correlated-gaussian}\ref{lmm:cdf} and the fact $\beta\lesssim \tau^2$ (see, e.g., \cite[Lemma 4]{carpentier2019minimax} or \eqref{eq-7} below) yields that
				\begin{align}
					\var\qth{\sum_{i\in \mathcal{U}} \s_i^2\beta \one(|\varepsilon_i|>\s_i\tau) }
					& = \beta^2 \sum_{i\in \mathcal{U}, j\in\mathcal{U}}\s_i^2\s_j^2 \cov(\one(|\tilde\varepsilon_i|>\tau), \one(|\tilde\varepsilon_j|>\tau))\\
					& \lesssim \|\bs{\S}\|_F^2 \tau^8 \exp(-\tau^2/3).\label{eq:truncate-moment-2}
				\end{align}
				The conclusion follows by combining \eqref{eq:truncate-moment-1} and \eqref{eq:truncate-moment-2}.
				
				\ref{prop:cosine} Introducing the notation $\cS=\{i\in[d]:\mu_i\ne 0\}$ we get
				\begin{align}
					\var\qth{\sum_{i=1}^d \cos(t (\mu_i+\varepsilon_i)) }
					& = \sum_{i,j}\cov\qth{ \cos(t (\mu_i+\varepsilon_i)), \cos(t (\mu_j+\varepsilon_j))  }\\
					& \overset{(a)}{\lesssim} \sum_{i,j} \tilde{\bs{\S}}_{ij}^2 + \sum_{i,j\in \cS} |\tilde{\bs{\S}}_{ij}| \\
					& \overset{(b)}{\le} \|\tilde{\bs{\S}}\|_F^2 + s \|\tilde{\bs{\S}}\|_F,
				\end{align}
				where $(a)$ follows from Lemma~\ref{lmm:correlated-gaussian}\ref{lmm:cosine}, and $(b)$ is due to the Cauchy-Schwarz inequality.
			\end{proof}

			\subsection{Proofs for Section \ref{sec:known-sigma} (known $\sigma$)}
			
			\subsubsection{Proof of Theorem \ref{thm:knownvar}}
			
			Consider first the case $s> \|\bs{\S}\|_F$. Then
			\[
			\hat{Q} - \|\bs{\t}\|_2^2
			= \sum_{i=1}^d Y_i^2 - \s^2\s_i^2 -\t_i^2
			= \s^2\sum_{i=1}^d (\varepsilon_i^2-\s_i^2) + 2\s \sum_{i\in\cS} \t_i \varepsilon_i,
			\]
			where $\cS$ denotes the support of $\bs{\t}$.
			Consequently,
			\begin{equation}
				\label{eq:knownvar-0}
				\esp\qth{ (\hat{Q} - \|\bs{\t}\|_2^2 )^2}
				\le 2 \s^4 \var (\| \bs{\varepsilon} \|_2^2) + 8 \s^2 \bs{\t}_{\cS}^{\top} \bs{\S}_{\cS, \cS} \bs{\t}_{\cS},
			\end{equation}
			where $\bs{\t}_{\cS}\in\RR^{\cS}$ is the vector of nonzero entries of $\bs{\t}$, and $\bs{\S}_{\cS, \cS}\in\RR^{\cS\times \cS}$ is the submatrix of $\bs{\S}$ with columns and rows in $\cS$.
			Applying Proposition~\ref{prop:correlated-gaussian}\ref{prop:square} we immediately obtain~\eqref{thmeq:hatq} for $s> \|\bs{\S}\|_F$.

			Consider now the case $s\le \|\bs{\S}\|_F$. Then we have 
			\begin{align}
				\hat Q - \|\bs{\t}\|_2^2
				&= \sum_{i=1}^d (Y_i^2 - \s^2\s_i^2\beta_s)\one(|Y_i|>\s\s_i\tau) - \t_i^2 \\
				&=  \sum_{i\not\in\cS} \s^2(\varepsilon_i^2 - \s_i^2\beta_s) \one(|\varepsilon_i| > \s_i\tau)
				+ \sum_{i\in\cS}(Y_i^2 - \s^2\s_i^2\beta_s - \t_i^2) 
				\\
				& \quad
				- \sum_{i\in\cS} (Y_i^2 - \s^2\s_i^2\beta_s)\one(|Y_i|\le\s\s_i\tau)\\
				&= \sum_{i=1}^d \s^2(\varepsilon_i^2 - \s_i^2\beta_s) \one(|\varepsilon_i| > \s_i\tau)  + \sum_{i\in\cS} 2\s \t_i\varepsilon_i  \label{eq:knownvar-1} \\
				& \qquad + \sum_{i\in\cS}\s^2(\varepsilon_i^2 - \s_i^2\beta_s)\one(|\varepsilon_i|\le\s_i \tau) - \sum_{i\in\cS} (Y_i^2 - \s^2\s_i^2\beta_s)\one(|Y_i|\le\s\s_i\tau). \label{eq:knownvar-2}
			\end{align}
			Notice that for any $\tau\ge 0$ we have  
			\begin{equation}\label{eq-7}
				\beta_s = 1+\frac{\tau\frac{1}{\sqrt{2\pi}}e^{-\tau^2/2}}{ \int_{\tau}^\infty \frac{1}{\sqrt{2\pi}}e^{-x^2/2}dx}
				< 1+\frac{\tau(\tau+\sqrt{\tau^2+4})}{2}
				\le \tau^2 + \tau + 1,   
			\end{equation}
			where the first inequality uses formula~7.1.13 in \cite{AS64}.
			Since $\tau\ge 1$, by the triangle inequality, we get that almost surely, the absolute value of the expression in~\eqref{eq:knownvar-2} satisfies:
			\[
			\abs{  \sum_{i\in\cS}(\varepsilon_i^2 - \s^2\s_i^2\beta_s)\one(|\varepsilon_i|\le\s_i\tau) - \sum_{i\in\cS} (Y_i^2 - \s^2\s_i^2\beta_s)\one(|Y_i|\le\s\s_i\tau)  }
			\lesssim  s \s^2  \tau^2 \max_{i\in\cS} \s_i^2.
			\]
			Next, the mean squared errors of the two terms in~\eqref{eq:knownvar-1} can be controlled by Proposition~\ref{prop:correlated-gaussian}\ref{prop:truncate-moments} and an argument similar to~\eqref{eq:knownvar-0}, respectively.
			It follows that
			\begin{align}
				&\esp\qth{\Big( \sum_{i=1}^d \s^2(\varepsilon_i^2 - \s_i^2\beta_s) \one(|\varepsilon_i| > \s_i\tau) \Big)^2}
				+ \esp\qth{\Big(\sum_{i\in\cS} 2\s \t_i\varepsilon_i )^2} \\
				\lesssim~& \s^4 \|\bs{\S}\|_F^2 \tau^4 e^{-\tau^2/2} + \s^2 \bs{\t}_{\cS}^{\top} \bs{\S}_{\cS, \cS} \bs{\t}_{\cS}.
			\end{align}
			Combining the above bounds we find that 
			\[
			\esp\qth{ (\hat{Q} - \|\bs{\t}\|_2^2)^2 }
			\lesssim \s^2 \|\bs{\t}\|_2^2 \lambda_{\max}(\bs{\S}_{\cS,\cS}) + \s^4 \tau^4 \pth{ s^2 + \|\bs{\S}\|_F^2 e^{-\tau^2/2}}.
			\]
			Recalling that $\tau=3\sqrt{  \log(1 + \|\bs{\S}\|_F^2 / s^2)  }$ finishes the proof of~\eqref{thmeq:hatq} for $s\le \|\bs{\S}\|_F$.

			Inequality \eqref{thmeq:hatq0} follows immediately from \eqref{thmeq:hatq} using the bounds
			\[
			\lambda_{\max}(\bs{\S}_{\cS,\cS})
			\le \lambda_{\max}(\bs{\S}) \wedge \|\bs{\S}_{\cS,\cS}\|_F
			\le \lambda_{\max}(\bs{\S}) \wedge s \max_{i\in \cS} \s_i^2 \le \lambda_{\max}(\bs{\S}) \wedge s.
			\]

			Finally, we deduce \eqref{thmeq:hatn} from \eqref{thmeq:hatq0}. Notice first that $\bar\psi(s,\bs{\S})\le \psi(s,\bs{\S})$.
			Introduce the threshold $t := \s \sqrt{\psi(s,\bs{\S})}$. Consider first the case $\|\bs{\t}\|_2 \le t$. Then, using the inequality $(a - b)^2 \le 2 (a^2 - b^2) + 4 b^2$, and \eqref{thmeq:hatq0} we obtain
			\begin{align}
				\esp[(\hat{N} - \|\bs{\t}\|_2)^2]
				& \le 2 \esp[|\hat{Q}| - \|\bs{\t}\|_2^2] + 4 \|\bs{\t}\|_2^2 \le 2 \sqrt{\esp[(\hat{Q} - \|\bs{\t}\|_2^2)^2]} + 4 t^2 \\
				& \lesssim \s^2 \psi(s,\bs{\S}) + \s t \sqrt{\bar\psi(s,\bs{\S})}  + t^2 \\
				& \lesssim \s^2 \psi(s,\bs{\S}).
			\end{align}
			Next, consider the case $\|\bs{\t}\|_2 > t$. Then, using the inequality $(a - b)^2 \le (a^2 - b^2)^2 / a^2$ for $a > 0, b \ge 0$, and \eqref{thmeq:hatq0} we get
			\begin{align}
				\esp[(\hat{N} - \|\bs{\t}\|_2)^2]
				& \le \frac{\esp [(|\hat Q| - \|\bs{\t}\|_2^2)^2]}{\|\bs{\t}\|_2^2}
				\le \frac{\esp[ (\hat Q - \|\bs{\t}\|_2^2)^2]}{t^2}\\
				& \lesssim \s^2 \bar \psi(s,\bs{\S}) +  \frac{\s^4 \psi^2(s,\bs{\S})}{t^2}\\
				& \lesssim \s^2 \psi(s,\bs{\S}).
			\end{align}
			The proof of \eqref{thmeq:hatn} is complete.
			
			\subsubsection{Proof of Theorem~\ref{thm:lower-known-sigma-rho}}
			
			Given any loss function $w(\cdot)$ with $w(0) = 0$ and $w \not\equiv 0$ and any $\phi > 0$, we define the scaled minimax risk as
			\[
			R^*(s, d, \rho, \phi) := \inf_{\hat{T}}\sup_{(\bs{\t}, \bs{\S}) \in \T(s, d, \rho)} \esp[w(\s^{-1} \phi^{-1/2}|\hat{T} - \|\bs{\t}\|_2|)].
			\]
			Since the diagonal elements of $\bs{\S}$ are all equal to 1, we only need to consider the case $\rho\ge \sqrt{d}$. The proof is based on a reduction of the minimax lower bound with correlated observations to that with independent observations. In particular, we define
			\[
			\tilde R^*(s, d, \phi) := \inf_{\hat{T}}\sup_{\bs{\t} \in B_0(s, d), \bs{\S} = \bs{I}_{d \times d}} \esp[w(\s^{-1} \phi^{-1/2}|\hat{T} - \|\bs{\t}\|_2|)].
			\]
			It is known from \cite{CCT2017} that, for any integers $d',s'$ such that $s'\in [d']$, we have the lower bound
			\begin{equation}\label{eq:rgrtsim1}
				\tilde R^*(s', d', \phi(s', d')) \gtrsim 1.   
			\end{equation}
			We will use this bound with suitably chosen $d',s'$ to prove that $R^*(s, d, \rho, \phi(s,\rho^2))\gtrsim 1$, which is the required result. We consider separately the following two cases.

			\paragraph{Case I: $\sqrt{d}\le \rho \le 2 \sqrt{d}$.} Since the identity matrix satisfies the Frobenius norm constraint in the definition of $\T(s, d, \rho)$ we have that, for any $w(\cdot)$ and $\phi$,
			\[
			R^*(s, d, \rho, \phi) \ge \tilde R^*(s, d, \phi).
			\]
			Set here $\phi=\phi(s, d)$ and note that, in the considered case, $\phi(s, d) \asymp\phi(s,\rho^2)$.  It follows immediately from \eqref{eq:rgrtsim1} with $s'=s$ and $d'=d$ that 
			$R^*(s, d, \rho, \phi(s,\rho^2)) \gtrsim 1$.
			
			
			\paragraph{Case II: $\rho > 2 \sqrt{d}$.}
			Without loss of generality we assume that $s \ge 2$.
			Given any two positive integers $r$ and $p$ such that
			\begin{equation}\label{eq:restr1}
				r p \le d,
			\end{equation}
			we define the following block diagonal covariance matrix:
			\begin{equation}\label{eq:sigma0}
				\bs{\S}_0 = \mathrm{diag}(\underset{p \;\mathrm{blocks}}{\underbrace{\bs{1}_{r \times r}, \cdots, \bs{1}_{r \times r}}}, \bs{I}_{(d - rp) \times (d - rp)}).
			\end{equation}
			Here $\bs{1}_{r \times r}$ is a $r \times r$ dimensional matrix with all entries equal to one, and $\bs{I}_{(d - rp) \times (d - rp)}$ is a $(d - rp) \times (d - rp)$ dimensional identity matrix. We also define a class of vectors $\bs{\t}$ as follows:
			\[
			\tilde{\T} := \left\{(\underset{r \;\mathrm{entries}}{\underbrace{\tilde{\t}_1, \cdots, \tilde{\t}_1}}, \cdots, \underset{r \;\mathrm{entries}}{\underbrace{\tilde{\t}_p, \cdots, \tilde{\t}_p}}, 0, \cdots, 0)^\top\in \RR^d: \sum_{i=1}^p \one (\tilde{\t}_i \neq 0) \le 1 \right\}.
			\]
			If the pair $(r,p)$ is such that
			\begin{equation}\label{eq:restr2}
				r \le s \quad\text{and}\quad r^2 p + d \le \rho^2,
			\end{equation}
			then for any $\bs{\bt}\in\tilde{\T}$, we have $(\bs{\t}, \bs{\S}_0) \in \T(s, d, \rho)$. Moreover, for $\bs{\t}\in\tilde{\T}$, $\bs{\S}=\bs{\S}_0$, model \eqref{eq:model} is equivalent to the model with $p$ independent observations $\tilde{Y}_i \sim \mathcal{N}(\tilde{\t}_i, \s^2)$, $i\in[p]$, such that at most one of $\tilde{\t}_i$'s is nonzero. 
			Since $\|\bs{\t}\|_2 = \sqrt{r} \sqrt{\sum_{i=1}^p \tilde{\t}_i^2}=\sqrt{r} \|\tilde{\bs{\t}}\|_2$, where $\tilde{\bs{\t}}=(\tilde \t_1,\dots,\tilde \t_p)$, we obtain that, for any $\phi>0$,
			\begin{align}\label{eq:r-to-r}
				& R^*(s, d, \rho, \phi) \ge \inf_{\hat{T}}\sup_{\tilde{\bs{\t}} \in B_0(1, p), \bs{\S} = \bs{I}_{p \times p}} \mathbf{E}_{\tilde{\bs{\t}}, \bs{\S}}\left[w\left(\frac{|\hat{T} - \sqrt{r} \|\tilde{\bs{\t}}\|_2|}{\s \phi^{1/2}}\right)\right] \\
				& \qquad\qquad\qquad = \inf_{\hat{T}}\sup_{\tilde{\bs{\t}} \in B_0(1, p), \bs{\S} = \bs{I}_{p \times p}} \mathbf{E}_{\tilde{\bs{\t}}, \bs{\S}}\left[w\left(\frac{|\hat{T} - \|\tilde{\bs{\t}}\|_2|}{\s (r^{-1}\phi)^{1/2}}\right)\right] = \tilde R^*(1, p, r^{-1} \phi).
			\end{align}
			Thus, to complete the proof it remains to find a pair $(r,p)$ satisfying the inequalities in~\eqref{eq:restr1},~\eqref{eq:restr2} and such that $\tilde R^*(1, p, r^{-1} \phi(s,\rho^2))\gtrsim 1$.
			We will do it separately for two cases. 
			
			First, consider the case $s \le \rho$. Then we choose $r$ and $p$ as
			the largest integers such that $r \le s / 2$ and $p \le \frac{\rho^2}{s^2} \wedge \frac{d}{s}$.
			Recalling that $\rho^2>4d$ one can check that such $r$ and $p$ satisfy \eqref{eq:restr1} and~\eqref{eq:restr2}.
			Furthermore,
			\begin{align}\label{eq:proof:th3}
				r \phi(1, p) &\asymp s \log\left(1 + \frac{\rho^2}{s^2} \wedge \frac{d}{s}\right) \asymp s \left(\log\left(1 + \frac{\rho^2}{s^2}\right) \wedge \log\left(1 + \frac{d}{s}\right)  \right)
				\\
				&\asymp s \log\left(1 + \frac{\rho^2}{s^2}\right) = \phi(s, \rho^2), 
			\end{align}
			where we have used the facts that $\rho \le d$, and $\log\left(1 + \frac{\rho^2}{s^2}\right) \asymp \log\left(1 + \frac{\rho}{s}\right)$ for $s\le \rho$. Combining the last two displays, we obtain that 
			there exists an absolute constant $c>0$ such that
			\[
			R^*(s, d, \rho, \phi(s, \rho^2)) \ge \tilde R^*(1, p, c \phi(1, p)) \gtrsim 1.
			\]
			Next, we consider the case $s>\rho$. Then $\phi(s, \rho^2)=\rho$. We choose the pair $(r,p)$ satisfying \eqref{eq:restr1} and~\eqref{eq:restr2} as follows: $r=\lfloor \rho/ 2 \rfloor$ and $p=1$. Clearly, $R^*(1, 1, r^{-1} \rho) \gtrsim 1$, which implies that $R^*(s, d, \rho, \phi(s, \rho^2)) \gtrsim 1$ also in this case.

			\subsection{Proofs for Section \ref{sec:sigma} (estimation of $\sigma^2$)}
			
			\subsubsection{Analysis of $\hat \s_S^2$}
			
			Proposition~\ref{prop:correlated-gaussian}\ref{prop:cdf} allows us to control the variance of the empirical distribution function $\hat F_{\varepsilon^2}(x)$ for any fixed $x\ge 0$. In addition, we have the bound \eqref{eq:FY-Feps}. Consequently, we obtain that, for a deterministic $t\asymp \sigma^2$, the estimator $\hat F_{Y^2}(t)$ converges to $F(t/\s^2)$ with the desired rate $\frac{s+\|\tilde{\bs{\S}}\|_F}{d}$.
			In our construction, we do not operate with fixed $t$ but rather with $\hat t$, which is data-driven and depends on the empirical distribution function, cf. \eqref{eq:def-hat-t}. However, since we use a dyadic grid in \eqref{eq:def-hat-t}, it is enough to control the estimation error for a finite number of values of $t$ to obtain the desired error guarantee for $\hat \s_S^2$. We first prove the bound on the error $\hat{\s}_S^2(t)$ for fixed $t$ as stated in Lemma~\ref{lmm:sigma-t-prob}, then derive the properties of $\hat t$ in Lemma~\ref{lmm:hat-t-property}, and finally obtain the error guarantee for $\hat \s_S^2$.

			\begin{proof}[Proof of Lemma~\ref{lmm:sigma-t-prob}]
				(i) We prove only the bound on the upper tail probability. The proof of the lower tail bound is entirely analogous.
				The definition of $\hat{\s}_S^2(t)$ and the monotonicity of $F$ imply that
				\begin{align}
					\prob\qth{\frac{\hat{\s}_S^2(t)}{\sigma^2}\ge 1+u}
					& =\prob\qth{ \hat F_{Y^2}(t) \le F\pth{\frac{t}{(1+u)\sigma^2}}}\\
					& \le \prob\qth{ F\pth{\frac{t}{\sigma^2}}-\hat F_{\varepsilon^2}\pth{\frac{t}{\sigma^2}}\ge F\pth{\frac{t}{\sigma^2}}- F\pth{\frac{t}{(1+u)\sigma^2}} - \frac{s}{d}  }
					, \label{eq-2}
				\end{align}
				where the last inequality uses \eqref{eq:FY-Feps}.
				Recall that $F(x)=2\Phi(\sqrt{x})-1$ for $x>0$.
				Then, for $|u| \le 1/2$,
				\begin{equation}\label{eq-1}
					\abs{
						F(x)-F(x/(1+u))
					}
					\ge \frac{2}{\sqrt{2\pi}} e^{-\frac{x}{2(1-|u|)}} \sqrt{x} \abs{1-\frac{1}{\sqrt{1+u}}}
					\ge \frac{1}{4} e^{-x}|u| \sqrt{x}.  
				\end{equation}
				Applying Chebyshev's inequality and Proposition~\ref{prop:correlated-gaussian}\ref{prop:cdf} we get
				\begin{align}
					\prob\qth{\frac{\hat{\s}_S^2(t)}{\sigma^2}\ge 1+u}
					& \le \frac{C\frac{ \|\tilde{\bs{\S}}\|_F^2  }{d^2} ( (\frac{t}{\s^2})^2 + 1 ) e^{-\frac{t}{3\sigma^2}} }{[(\frac{1}{4}|u|e^{-\frac{t}{\sigma^2}} (\frac{t}{\sigma^2})^{1/2} - \frac{s}{d})_+]^2}.
				\end{align}
				Part (i) of the lemma follows from this inequality by noticing that $
				1\wedge \frac{x}{[(y-a)_+]^2} \le 1 \wedge \frac{2 (x + a^2)}{y^2}$ for any $x, y, a \ge 0$. Indeed, $1\wedge \frac{x}{[(y-a)_+]^2} = (1 \wedge \frac{\sqrt{x}}{(y-a)_+})^2$, and $1\wedge \frac{\sqrt{x}}{(y-a)_+} = \one\{y \le \sqrt{x} + a\} + \frac{\sqrt{x}}{y - a}\one\{y > \sqrt{x} + a\} \le 1 \wedge \frac{\sqrt{x} + a}{y}$.
				%
				
				We now prove part (ii) of the lemma. For any $t$ satisfying $|t / \s^2 - 1| \le 2 / 3$ and $u=1/2$ we obtain from \eqref{eq-1} that
				\[
				\abs{
					F(t / \s^2)-F(2t/3\s^2))
				}
				\ge \frac{2}{\sqrt{2\pi}} e^{-\frac{5}{3}} \left(\sqrt{\frac{1}{3}} - \sqrt{\frac{2}{9}}\right) \ge 0.015.
				\]
				In view of this remark, part (ii) of the lemma follows from \eqref{eq-2} by applying Chebyshev's inequality and Proposition~\ref{prop:correlated-gaussian}\ref{prop:cdf}.
			\end{proof}
			
			\begin{lemma}\label{lem:quantile}
				For $\alpha \in (0,1)$, let $\hat{q}_\alpha$ be the empirical $\alpha$-th quantile of $\tilde\varepsilon_i^2$'s, that is,
				\[
				\hat{q}_\alpha := \min \{u: \hat{F}_{\varepsilon^2}(u) \ge \alpha \}.
				\]
				Then
				\[
				\esp(\hat{q}_\alpha^2) \le 1 + \frac{8}{(1-\alpha)\sqrt{2\pi}}.
				\]
			\end{lemma}
			
			\begin{proof} Since $\hat{F}_{\varepsilon^2}(\cdot)$ is the empirical cdf of $\tilde\varepsilon_i^2$'s, where $\tilde\varepsilon_i$'s are standard normal variables, for any $t>0$ we have
				\[
				\esp[1-\hat{F}_{\varepsilon^2}(t)] = \prob(\tilde\varepsilon_1^2 > t) \le \frac{2e^{-\frac{t}2}}{\sqrt{2 \pi t}}.
				\]
				By Markov's inequality we obtain that, for any $t > 0$,
				\begin{align}
					\prob(\hat{q}_\alpha > t) \le \prob(\hat{F}_{\varepsilon^2}(t) \le \alpha) = \prob(1 - \hat{F}_{\varepsilon^2}(t) \ge 1 - \alpha) \le \frac{2e^{-\frac{t}{2}}}{(1- \alpha)\sqrt{2\pi t}}.
				\end{align}
				Hence
				\[
				\esp[\hat{q}_\alpha^2]
				\le 1 + \int_1^{\infty} \prob(\hat{q}_\alpha^2 > t) dt \le 1 + \int_1^{\infty} \frac{2e^{-\frac{\sqrt{t}}{2}}}{(1-\alpha)\sqrt{2\pi }t^{1/4}}dt \le 1 + \frac{8}{(1-\alpha)\sqrt{2\pi}}.
				\]
			\end{proof}

			\begin{lemma}
				\label{lmm:hat-t-property}
				Let $s/d\le 0.01$.
				Then there exists an absolute constant $c > 0$ such that
				\begin{equation}\label{eq-3}
					\prob\qth{
						{\frac{1}{3}}\sigma^2
						\le \hat t\le \frac{3}{2}\sigma^2
					} \ge 1-\frac{c\|\tilde{\bs{\S}}\|_F^2}{d^2}.  
				\end{equation}
				Furthermore, $\esp[\hat t^2]\lesssim \sigma^4$.	
			\end{lemma}
			\begin{proof}
				Since $F$ is the cdf of the $\chi^2_1$ distribution we have $F(\alpha)+\tau<0.5<F(\beta)-\tau$ for $\alpha=\frac{1}{3}$, $\beta=\frac{3}{4}$ and $\tau=0.01$.
				By Chebyshev's inequality and Proposition~\ref{prop:correlated-gaussian}\ref{prop:cdf},
				\begin{align}
					\prob[\hat F_{\varepsilon^2}(\beta )-\tau \le 0.5]
					& \le \frac{\esp \qth{|\hat F_{\varepsilon^2}( \beta ) - F(\beta) |^2}}{|F(\beta )-\tau-0.5|^2}
					\lesssim \frac{\|\tilde{\bs{\S}}\|_F^2}{d^2},\\
					\prob[\hat F_{\varepsilon^2}(\alpha )+\tau \ge 0.5]
					& \le \frac{\esp \qth{|\hat F_{\varepsilon^2}( \alpha ) - F(\alpha) |^2}}{|F(\alpha )+\tau-0.5|^2}
					\lesssim \frac{\|\tilde{\bs{\S}}\|_F^2}{d^2}.
				\end{align}
				By the definition of $\hat t$, we have $\hat F_{Y^2}(\hat t)\ge 0.5$ and $
				\hat F_{Y^2}(\hat t/2)< 0.5.$
				Combining the above remarks and~\eqref{eq:FY-Feps} we obtain that, with probability at least $1-c \|\tilde{\bs{\S}}\|_F^2 / d^2$, where $c>0$ is an absolute constant,
				\begin{align}
					& \hat F_{Y^2}(\hat t)\ge 0.5 > \hat F_{\varepsilon^2}(\alpha )+\tau
					\ge \hat F_{Y^2}(\alpha \sigma^2),\\
					& \hat F_{Y^2}(\hat t/2)< 0.5 < \hat F_{\varepsilon^2}(\beta )-\tau
					\le \hat F_{Y^2}(\beta \sigma^2).
				\end{align}
				The bound \eqref{eq-3} follows from these relations and the monotonicity of $\hat F_{Y^2}$.
				Next, from~\eqref{eq:FY-Feps} and the assumption $s/d\le \tau$ we obtain that $\hat F_{\varepsilon^2}(\frac{\hat t}{2\sigma^2})
				\le \hat F_{Y^2}\pth{\hat t /2} + \tau
				<0.6
				$.
				It means that
				$\frac{\hat{t}}{2 \s^2}\le \hat{q}_{0.6}$,  where $\hat{q}_{0.6}$ is the $0.6$-th empirical quantile of $\tilde\varepsilon_i^2$'s. Applying Lemma~\ref{lem:quantile} we conclude that $\esp[\hat t^2]\lesssim \sigma^4$. 
			\end{proof}

			Consider the set
			\[
			\calT
			:= \sth{ t=2^{\ell}: \ell\in \ZZ, \frac{\sigma^2}{3}\le t \le \frac{3}{2}\sigma^2  }.
			\]
			Note that this set contains not more than 3 elements. Indeed, if $\calT$ is non-empty and $\ell^*$ is the smallest value in $\ZZ$ such that $\frac{\sigma^2}{3}\le 2^{\ell^*} \le \frac{3}{2}\sigma^2$ then $2^{\ell^*+3}\ge 2^4 {\sigma^2}/3 > \frac{3}{2}\sigma^2$.
			
			\Cref{lmm:hat-t-property} implies that there is an absolute constant $c>0$ such that $\hat t \in \calT$ with probability at least $1 - c \|\tilde{\bs{\S}}\|_F^2 / d^2$. 
			Before proving \Cref{thm:hat-sigma-sparse}, we first show that there exists an absolute constant $c>0$ such that $\hat\sigma_S^2\asymp \sigma^2$ with probability at least $1 - c\|\tilde{\bs{\S}}\|_F^2 / d^2$.

			\begin{lemma}
				\label{lmm:hat-sigma-whp}
				There exists an absolute constant $c > 0$ such that
				\[
				\prob\qth{\frac{1}{2}\sigma^2\le \hat \sigma_S^2 \le \frac{3}{2}\sigma^2}
				\ge 1- c \frac{\|\tilde{\bs{\S}}\|_F^2}{d^2}.
				\]
			\end{lemma}
			\begin{proof}
				Recall that, by definition, $\hat\sigma_S^2 = \hat{\s}_S^2(\hat t)$.
				We have
				\begin{align}
					\prob\qth{ \abs{\frac{\hat\sigma_S^2}{\sigma^2}-1}>\frac{1}{2} }
					& \le \prob\qth{ \abs{\frac{\hat{\s}_S^2(\hat t)}{\sigma^2}-1}>\frac{1}{2}, \hat t\in \calT} + c \frac{\|\tilde{\bs{\S}}\|_F^2}{d^2}\\
					& \le \sum_{t\in\calT }\prob\qth{ \abs{\frac{\hat{\s}_S^2(t)}{\sigma^2}-1}>\frac{1}{2}} + c \frac{\|\tilde{\bs{\S}}\|_F^2}{d^2}.
				\end{align}
				The conclusion follows from \Cref{lmm:sigma-t-prob}(ii) and the fact that $\calT$ contains not more than 3 elements.
			\end{proof}
			
			\begin{proof}[Proof of \Cref{thm:hat-sigma-sparse}]
				Consider the event $E:= \{ \frac{1}{2}\sigma^2 \le \hat \sigma_S^2 \le \frac{3}{2}\sigma^2, \hat t\in\calT \}$. \Cref{lmm:hat-t-property,lmm:hat-sigma-whp} imply that  $\prob(E)\ge 1-c \frac{\|\tilde{\bs{\S}}\|_F^2}{d^2}$, where $c>0$ is an absolute constant.
				We have
				\[
				\esp\abs{\frac{\hat \sigma_S^2}{\sigma^2}-1}
				=\esp\left(\abs{\frac{\hat \sigma_S^2}{\sigma^2}-1}\one(E)\right)
				+\esp\left(\abs{\frac{\hat \sigma_S^2}{\sigma^2}-1}\one(E^c)\right).
				\]
				By the definition of $\hat t$ we have $\hat F_{Y^2}(\hat t)\ge 0.5$. Thus, 
				\begin{equation}\label{eq:sigscale}
					\frac{\hat \sigma_S^2}{\sigma^2}
					=\frac{\hat t}{\sigma^2 F^{-1}( \hat F_{Y^2}(\hat t) ) }
					\lesssim \frac{\hat t}{\sigma^2}.
				\end{equation}
				Applying the Cauchy-Schwarz inequality and \Cref{lmm:hat-t-property} yields
				\begin{equation}
					\label{eq:sig-S-const}
					\esp\qth{\abs{\frac{\hat \sigma_S^2}{\sigma^2}-1}\one(E^c)}
					\lesssim \frac{\|\tilde{\bs{\S}}\|_F}{d} \sqrt{1+\esp\qth{(\hat t/\sigma^2)^2}}
					\lesssim \frac{\|\tilde{\bs{\S}}\|_F}{d}.
				\end{equation}
				Next, note that
				\begin{align}
					\abs{
						F(\hat t/\sigma^2) - F(\hat t/\hat\sigma_S^2)
					}\one(E)
					\ge \min_{\xi\in[{\frac{2}{9}},3]}|F'(\xi)| \, \abs{\frac{\hat t}{\sigma^2} - \frac{\hat t}{\hat\sigma_S^2}}\one(E)
					\gtrsim \abs{\frac{\hat t}{\sigma^2} - \frac{\hat t}{\hat\sigma_S^2}}\one(E).
				\end{align}
				Using this inequality and \eqref{eq:sigscale} we find 
				\begin{align}
					&\phantom{{}={}}\esp\qth{\abs{\frac{\hat \sigma_S^2}{\sigma^2}-1}\one(E)}
					=\esp\qth{\frac{\hat\sigma_S^2}{\hat t}\abs{\frac{\hat t}{\sigma^2} - \frac{\hat t}{\hat\sigma_S^2}} \one(E) } \\
					& \lesssim  \esp\qth{\abs{F(\hat t/\sigma^2) - F(\hat t/\hat\sigma_S^2)}\one(E)}
					= \esp\qth{\abs{F(\hat t/\sigma^2) - \hat F_{Y^2}(\hat t) }\one(E)} \\
					& \le \sum_{t\in\calT}\esp\abs{F(t/\sigma^2) - \hat F_{Y^2}(t) }
					\lesssim \frac{s+\|\tilde{\bs{\S}}\|_F}{d}\label{eq:ub-mse},
				\end{align}
				where the last step follows from Proposition~\ref{prop:correlated-gaussian}\ref{prop:cdf}, \eqref{eq:FY-Feps} and the fact that $\calT$ contains at most 3 elements. Combining \eqref{eq:sig-S-const} and \eqref{eq:ub-mse} concludes the proof.
			\end{proof}

			\subsubsection{Analysis of $\hat \s^2_D$}
			
			In this subsection, we prove \Cref{thm:hat-sigma-dense}. We first evaluate the accuracy of  $\hat\varphi_d(\varz)$ considered as an estimator of the normal characteristic function $e^{-\frac{\varz^2\sigma^2}{2}}$. The following lemma
			gives a bound on the relative error of $\hat\varphi_d(\varz)$. 

			\begin{lemma}
				\label{lmm:phi-multiplicative}
				For any $z\in\mathbb{R}$ we have
				\begin{equation}
					\label{eq:phi-MSE}
					\esp\qth{\pth{\frac{\hat\varphi_d(\varz)}{e^{-\frac{\varz^2\sigma^2}{2}}}-1}^2}
					\lesssim \frac{s^2}{d^2} + \frac{\|\tilde{\bs{\S}}\|_F^2}{d^2}\pth{\frac{s}{\|\tilde{\bs{\S}}\|_F}\vee 1} \exp\pth{\varz^2\sigma^2}.
				\end{equation}
				Furthermore, if $s\le {d}/{8}$ there exists an absolute constant $C>0$ such that, for any $z\in\mathbb{R}$,
				\begin{equation}
					\label{eq:high-prob-phi}
					\prob\qth{ \frac{1}{2}e^{-\frac{\varz^2\s^2}{2}} \le \hat\varphi_d(\varz) \le \frac{3}{2}e^{-\frac{\varz^2\s^2}{2}}  }
					\ge 1- C\frac{\|\tilde{\bs{\S}}\|_F^2}{d^2}\pth{\frac{s}{\|\tilde{\bs{\S}}\|_F}\vee 1} \exp\pth{\varz^2\sigma^2}.
				\end{equation}
			\end{lemma}
			\begin{proof}
				Applying \eqref{eq:expect-cos} yields
				\begin{align}\label{eq-4}
					\abs{ \frac{\esp[\hat\varphi_d(\varz)]}{e^{-\frac{\varz^2\sigma^2}{2}}}-1 }
					=  \abs{ \frac{1}{d}\sum_{i=1}^d\cos(\varz \t_i) - 1 }
					\le \frac{2s}{d}.
				\end{align}
				Consequently,
				\[
				\esp\qth{\pth{\frac{\hat\varphi_d(\varz)}{e^{-\frac{\varz^2\sigma^2}{2}}}-1}^2}
				\le 2\esp\qth{\pth{\frac{\hat\varphi_d(\varz)-\esp\qth{\hat\varphi_d(\varz)}}{e^{-\frac{\varz^2\sigma^2}{2}}}}^2} + 2 \pth{\frac{2s}{d}}^2.
				\]
				The bound \eqref{eq:phi-MSE} follows this inequality and Proposition~\ref{prop:correlated-gaussian}\ref{prop:cosine}.
				
				If $s\le {d}/{8}$ then \eqref{eq-4} implies that
				\begin{align}
					\abs{ \esp[\hat\varphi_d(\varz)] - e^{-\frac{\varz^2\sigma^2}{2}} }
					\le e^{-\frac{\varz^2\sigma^2}{2}}/4.
				\end{align}
				Under this condition, the left hand side of \eqref{eq:high-prob-phi} is not smaller than the probability
				\[
				\prob\qth{ |\hat\varphi_d(\varz) - \esp[\hat\varphi_d(\varz)] | \le e^{-\frac{\varz^2\sigma^2}{2}} /4  }
				\ge 1 - 16 e^{\varz^2\sigma^2} \var[\hat\varphi_d(\varz)]
				.
				\]
				Using this remark and Proposition~\ref{prop:correlated-gaussian}\ref{prop:cosine} we obtain \eqref{eq:high-prob-phi}.
			\end{proof}

			\begin{proof}[Proof of \Cref{thm:hat-sigma-dense}]
				Recall that $\hat\s_D^2=\min\{\tilde\s_D^2, 2\hat\s_S^2\}$, where
				\[
				\tilde\s_D^2
				= -\frac{2\hat t}{\lambda}\log |\hat\varphi_d( (\lambda/\hat t)^{1/2}  )|
				\]
				and $\lambda=\frac{1\vee \log(s/\|\tilde{\bs{\S}}\|_F)}{6}$.
				
				Define the event $E=\{\hat \s_S^2\ge \frac{1}{2}\s^2, \hat t\in\calT,  \frac{1}{2}e^{-\frac{\lambda \s^2}{2\hat t}}\le \hat\varphi_d( (\lambda/\hat t)^{1/2} ) \le \frac{3}{2}e^{-\frac{\lambda \s^2}{2\hat t}} \}$.
				Then,
				\begin{align}
					\esp\qth{\abs{ \frac{\hat\s_D^2}{\s^2}-1  }  }
					= \esp\qth{\abs{ \frac{\hat\s_D^2}{\s^2}-1  }\one(E)  }
					+ \esp\qth{\abs{ \frac{\hat\s_D^2}{\s^2}-1  }\one(E^c)  }.  \label{eq:hat-s-D-decompose}
				\end{align}
				Next, we bound separately the two terms on the right hand side of \eqref{eq:hat-s-D-decompose}.

				On the event $E$ we have $\s^2\le 2\hat\s_S^2$, which together with the fact that $\hat\s_D^2 = \min\{\tilde\s_D^2, 2\hat\s_S^2\}$ implies that $|\hat\s_D^2-\s^2|\one(E) \le |\tilde\s_D^2-\s^2| \one(E)$.
				Using the inequality $|\log x-\log y|\le \frac{|x-y|}{x\wedge y}$, $\forall x,y> 0$, we obtain:
				\begin{align}
					\abs{ \frac{\hat\s_D^2}{\s^2}-1  }\one(E)
					& \le \abs{ \frac{\hat\s_D^2(\hat t)}{\s^2}-1  }\one(E)
					= \frac{1}{\s^2}\abs{ -\frac{2\hat t}{\lambda}\pth{\log|\hat \varphi_d( (\lambda/\hat t)^{1/2}  )| - \log e^{-\frac{\lambda \sigma^2}{2\hat t}}  }}\one(E)
					\\
					& \le \frac{2\hat t}{\lambda \s^2}    \frac{|\hat\varphi_d( (\lambda/\hat t)^{1/2}  ) - e^{-\frac{\lambda \sigma^2}{2\hat t}}  |}{\hat\varphi_d( (\lambda/\hat t)^{1/2}) \wedge  e^{-\frac{\lambda \sigma^2}{2\hat t}} }   \one(E)
					\\
					& \le \frac{6}{\lambda} \abs{ \frac{\hat\varphi_d( (\lambda/\hat t)^{1/2}  )}{e^{-\frac{\lambda \sigma^2}{2\hat t}}} - 1  }   \one(E),
				\end{align}
				where the last inequality follows from the conditions $\hat t/\sigma^2\le 3/2$ and $\hat\varphi_d( (\lambda/\hat t)^{1/2})\ge e^{-\frac{\lambda \sigma^2}{2\hat t}}/2$ that hold on~$E$.
				Using the fact that $\hat t\in \calT$ under $E$, where $\calT$ contains at most 2 elements, and applying \Cref{lmm:phi-multiplicative} we find:
				\begin{align}
					\esp\qth{\abs{ \frac{\hat\s_D^2}{\s^2}-1  }\one(E)}
					& \le \sum_{t\in\calT} \frac{6}{\lambda} \esp\abs{ \frac{\hat\varphi_d( (\lambda/ t)^{1/2}  )}{e^{-\frac{\lambda \sigma^2}{2 t}}}-1  }
					\\
					& \lesssim \frac{1}{\lambda}\max_{t\in \calT} \pth{\frac{s}{d} + \frac{\|\tilde{\bs{\S}}\|_F}{d}\exp\pth{ \frac{\lambda\sigma^2}{2t}}\sqrt{\frac{s}{\|\tilde{\bs{\S}}\|_F}\vee 1} } 
					\\
					& \le \frac{s}{\lambda d} +  \frac{ e^{3\lambda/2} \|\tilde{\bs{\S}}\|_F}{\lambda d}\sqrt{\frac{s}{\|\tilde{\bs{\S}}\|_F}\vee 1}.\label{eq:dense-decompose-ub1}
				\end{align}
				We now consider the second term on the right hand side of \eqref{eq:hat-s-D-decompose}. Note first that since $\calT$ contains at most 3 elements and $\s^2/3 \le t\le 3\s^2/2$ for $t\in\calT$, we get by applying \Cref{lmm:phi-multiplicative} that
				\[
				\prob\qth{ \frac{1}{2}e^{-\frac{\lambda \s^2}{2 t}}\le \hat\varphi_d( (\lambda/ t)^{1/2} ) \le \frac{3}{2}e^{-\frac{\lambda \s^2}{2 t}},\forall t\in\calT   }
				\ge 1-C\frac{\|\tilde{\bs{\S}}\|_F^2}{d^2}\pth{\frac{s}{\|\tilde{\bs{\S}}\|_F}\vee 1} \exp\pth{3\lambda}.
				\]
				Combining this inequality with \Cref{lmm:hat-t-property,lmm:hat-sigma-whp} we get
				\begin{align}\label{eq-5}
					\prob[E^c]\lesssim \frac{\|\tilde{\bs{\S}}\|_F^2}{d^2}\pth{\frac{s}{\|\tilde{\bs{\S}}\|_F}\vee 1} \exp\pth{3\lambda}.
				\end{align}
				Using the fact that $\hat\s_D^2\le 2\hat \s_S^2$ and the Cauchy-Schwarz inequality we find:
				\begin{align}
					\esp\qth{\abs{ \frac{\hat\s_D^2}{\s^2}-1  }\one(E^c) }
					&\le \sqrt{\esp\qth{\pth{ \frac{\hat\s_D^2}{\s^2}-1   }^2}} \sqrt{\prob[E^c]}
					\\
					&\le \sqrt{2+2\esp\qth{(2\hat \s_S^2/\s^2)^2}}\sqrt{\prob[E^c]}. \label{eq:dense-decompose-ub2}
				\end{align}
				From Lemma~\ref{lmm:hat-t-property} and \eqref{eq:sigscale} we have $\esp\qth{(\hat \s_S^2/\s^2)^2}\lesssim 1$. Using this remark and combining \eqref{eq:hat-s-D-decompose} \eqref{eq:dense-decompose-ub1}, \eqref{eq-5}, and 
				\eqref{eq:dense-decompose-ub2} we obtain
				\begin{equation}\label{eq-6}
					\esp\qth{\abs{ \frac{\hat\s_D^2}{\s^2}-1  }  }
					\lesssim \frac{s}{\lambda d} +  \left(1+\frac{1}{\lambda}\right) \frac{ e^{3\lambda/2} \|\tilde{\bs{\S}}\|_F}{ d }\pth{\frac{s}{\|\tilde{\bs{\S}}\|_F}\vee 1}^{1/2}. 
				\end{equation}
				Recall that $\lambda = \big(1 \vee \log \frac{s}{\|\tilde{\bs{\S}}\|_F}\big)/6$. If $s\le e\|\tilde{\bs{\S}}\|_F$ the right hand side of \eqref{eq-6} does not exceed $C\|\tilde{\bs{\S}}\|_F/d$ for an absolute constant $C>0$. In the case $s> e\|\tilde{\bs{\S}}\|_F$ we have $\lambda = \big(\log \frac{s}{\|\tilde{\bs{\S}}\|_F}\big)/6$ and the right hand side of \eqref{eq-6} is of the order $\frac{s}{\lambda d}$ since 
				$$
				\frac{ e^{3\lambda/2} \|\tilde{\bs{\S}}\|_F}{ d }\pth{\frac{s}{\|\tilde{\bs{\S}}\|_F}\vee  
					1}^{1/2} \le  \frac{s}{ \big(s/\|\tilde{\bs{\S}}\|_F\big)^{1/4} d}
				\le \frac{s}{\lambda d}.
				$$
				This completes the proof.
			\end{proof}

			\subsection{Proofs for Section \ref{sec:unknown-sigma} (unknown $\sigma$)}
			
			In this section, we prove Theorem~\ref{thm:unknownvar1} using the estimators of the noise level $\s^2$ constructed in Section~\ref{sec:sigma}. Under Assumptions~\ref{as:var0} and~\ref{as:var1}, we have $\bs{\S}_{ij}/(\s_i\s_j) \asymp \bs{\S}_{ij}$ and thus $\|\bs{\S}\|_F \asymp \|\tilde{\bs{ \S}}\|_F$. Hence, we can replace $\|\tilde{\bs{ \S}}\|_F$ by $\|\bs{\S}\|_F$ in the theorems of Section~\ref{sec:sigma} and the bounds on estimation errors remain valid to within a constant factor.
			
			\subsubsection{Proof of part (i) of Theorem~\ref{thm:unknownvar1}}
			For $s > \| \bs{\S} \|_F$  we have
			\begin{align}
				\abs{\hat{N}^* - \|\bt\|_2}
				&= \abs{\sqrt{\abs{\sum_{i=1}^d Y_i^2-\hat \s_D^2 \s_i^2 }} - \|\bs{\t}\|_2}
				\\
				& \le
				\abs{ \sqrt{ \abs{ \|\bs{\t}\|_2^2 + 2\s \bt^\top \bs{\varepsilon} } } - \|\bs{\t}\|_2}
				+\sqrt{\abs{ \sum_{i=1}^d \s^2\varepsilon_i^2 - \hat \s_D^2\s_i^2 }},
				\label{eq:hat-Nstar-dense-1}
			\end{align}
			where the upper bound is due to the inequality $|\sqrt{|x+y|}-z|\le |\sqrt{|x|}-z|+\sqrt{|y|}$.
			Due to the inequality $|\sqrt{|1+x|}-1|\le |x|$, the first term on the right hand side of \eqref{eq:hat-Nstar-dense-1} admits the bound
			\[
			\abs{\sqrt{ \abs{\|\bs{\t}\|_2^2 + 2\s \bt^\top \bs{\varepsilon}} } - \|\bs{\t}\|_2}
			\le \frac{2 \s |\bt^\top \bs{\varepsilon}|}{\|\bt\|_2}.
			\]
			Since ${\mathbf E}(\bs{\varepsilon} \bs{\varepsilon}^\top)=\bs{\S}$
			we obtain
			\begin{equation}
				\label{eq:hat-Nstar-dense-2}
				{\mathbf E} \abs{\sqrt{ \abs{\|\bs{\t}\|_2^2 + 2\s \bt^\top \bs{\varepsilon}} } - \|\bs{\t}\|_2}^2
				\le 4 \s^2 \lambda_{\max} (\bs{\S})
				\le 4 \s^2 \| \bs{\S} \|_F.
			\end{equation}
			To handle the second term on the right hand side of \eqref{eq:hat-Nstar-dense-1} we first write
			\[
			\abs{ \sum_{i=1}^d \Big(\varepsilon_i^2 -  \frac{\hat \s_D^2}{\s^2}\s_i^2\Big) }
			\le \abs{ \sum_{i=1}^d (\varepsilon_i^2 - \s_i^2) }
			+
			\abs{ \sum_{i=1}^d \s_i^2 \pth{1 - \frac{\hat \s_D^2}{\s^2} }}.
			\]
			Assumption~\ref{as:var0} implies that $\sum_{i=1}^d \s_i^2\le d$. Therefore, we can use Theorem~\ref{thm:hat-sigma-dense} to get
			\begin{equation}
				\esp\abs{ \sum_{i=1}^d \s_i^2 \pth{1 - \frac{\hat \s_D^2}{\s^2} }}
				\lesssim \| \bs{\S} \|_F +
				\frac{s}{1 \vee \log\frac{s}{\|{\bs{\S}}\|_F} }.
			\end{equation}
			Furthermore, $\esp \abs{ \sum_{i=1}^d (\varepsilon_i^2 - \s_i^2) }\le \sqrt{2}\|{\bs{\S}}\|_F$ by
			Proposition~\ref{prop:correlated-gaussian}~\ref{prop:var-chi}. It follows that
			\begin{equation}
				\label{eq:hat-Nstar-dense-3}
				\esp\abs{ \sum_{i=1}^d \Big(\varepsilon_i^2 -  \frac{\hat \s_D^2}{\s^2}\s_i^2\Big) }
				\lesssim \| \bs{\S} \|_F +
				\frac{s}{1 \vee \log\frac{s}{\|{\bs{\S}}\|_F} }
				\lesssim \psi^*(s, \bs{\S}).
			\end{equation}
			Combining \eqref{eq:hat-Nstar-dense-1}, \eqref{eq:hat-Nstar-dense-2}, and \eqref{eq:hat-Nstar-dense-3} we complete the proof.

			\subsubsection{Proof of parts (ii) and (iii) of Theorem~\ref{thm:unknownvar1}}
			
			Using the definition of $\hat N^*$  for $s \le \| \bs{\S} \|_F$ we obtain
			\begin{align}
				\label{eq:hat-Nstar-sparse-1}
				|\hat{N}^* - \|\bt\|_2|
				&\le \left|\sqrt{\sum_{i \in \cS} Y_i^2 \one(|Y_i| >  \hat{\s}_S \s_{i} \tau)} - \|\bt\|_2\right|
				\\
				&\quad + \s \left|\sum_{i \not\in \cS} \varepsilon_i^2 \one\left(|\varepsilon_i| > \frac{\hat \s_S}{\s}\s_{i}\tau \right) - \alpha_s \frac{\hat{\s}_S^2}{\s^2} \, \sum_{i=1}^d \s_{i}^2\right|^{1 / 2}.
			\end{align}
			Here
			\begin{align}
				\left|\sqrt{\sum_{i \in \cS} Y_i^2 \one(|Y_i| >  \hat{\s}_S \s_{i} \tau)} - \|\bt\|_2\right|
				& \le \sqrt{\sum_{i\in\cS} \pth{Y_i \one(|Y_i| >  \hat{\s}_S \s_{i} \tau) - \t_i}^2 }\\
				& = \sqrt{ \sum_{i\in\cS} \pth{ \s \varepsilon_i - Y_i\one(|Y_i| \le \hat{\s}_S \s_{i} \tau) }^2 }.
			\end{align}
			Therefore,
			\begin{align}
				\esp \qth{\left|\sqrt{\sum_{i \in \cS} Y_i^2 \one(|Y_i| >  \hat{\s}_S \s_{i} \tau)} - \|\bt\|_2\right|^2}
				&\le 2\esp \pth{ \sum_{i\in\cS} \s^2 \varepsilon_i^2  + \hat{\s}_S^2 \s_{i}^2 \tau^2 }
				\\ &= 2  \pth{\s^2+ \esp (\hat{\s}_S^2) \tau^2} \sum_{i\in\cS} \s_{i}^2.    
			\end{align}
			Note that since $\|\tilde{\bs{\S}}\|_F\le d$ it follows from Theorem~\ref{thm:hat-sigma-sparse} 
			that $\esp (\hat{\s}_S^2)\lesssim \s^2$. Moreover, Assumption~\ref{as:var0} implies that $\sum_{i\in\cS} \s_{i}^2 \le s$. It follows that
			\begin{equation}
				\label{eq:hat-Nstar-sparse-2}
				\esp \qth{\left|\sqrt{\sum_{i \in \cS} Y_i^2 \one(|Y_i| >  \hat{\s}_S \s_{i} \tau)} - \|\bt\|_2\right|^2}
				\lesssim s\s^2\tau^2
				\lesssim \s^2 \psi^*(s,\bs{\S}).
			\end{equation}

			The remaining proof is devoted to the control of the second term on the right hand side of~\eqref{eq:hat-Nstar-sparse-1}.
			The major difficulty arises from the dependence between the variance estimator $\hat\s_S$ and $\varepsilon_i$'s.
			To begin with, we note that the following holds:
			\begin{align}
				\esp\left|\sum_{i \not\in \cS} \varepsilon_i^2 \one\left(|\varepsilon_i| > \s_{i}\tau \right) - \alpha_s \frac{\hat{\s}_S^2}{\s^2} \, \sum_{i=1}^d \s_{i}^2\right|
				& \le \esp\left|\sum_{i \not\in \cS} \varepsilon_i^2 \one\left(|\varepsilon_i| > \s_{i}\tau \right) - \alpha_s  \sum_{i=1}^d \s_{i}^2\right|
				\\
				& \quad + \alpha_s  \sum_{i=1}^d \s_{i}^2\esp\abs{\frac{\hat{\s}_S^2}{\s^2}-1} \\
				& \overset{(a)}
				{\le} 
				\esp\left|\sum_{i \not\in \cS} \big(\varepsilon_i^2 \one\left(|\varepsilon_i| > \s_{i}\tau \right) - \alpha_s \s_{i}^2\big)\right|
				\\
				&\quad + C \alpha_s (s+\|\bs{\S}\|_F)
				\\
				& \overset{(b)}{\lesssim} \|\bs{\S}\|_F \, \tau^4 \exp(-\tau^2/6)
				\lesssim  \psi^*(s,\bs{\S}), 
				\label{eq:hat-Nstar-sparse-3a}
			\end{align}
			where (a) follows from Assumption~\ref{as:var0} and Theorem~\ref{thm:hat-sigma-sparse} while
			(b) is due to the second inequality in Proposition~\ref{prop:correlated-gaussian}\ref{prop:truncate-moments},  the assumption that $s\le \|\bs{\S}\|_F$, the inequality $\alpha_s\lesssim \tau^2 e^{-\tau^2/2}$ and the fact that $\tau=3\sqrt{\log (1 + \|\bs{\S}\|_F^2 / s^2)}$. 
			
			Given \eqref{eq:hat-Nstar-sparse-3a}, in order to evaluate the second term on the right hand side of~\eqref{eq:hat-Nstar-sparse-1} it remains to control the difference
			\begin{equation}
				\label{eq:eps-hat-sigma}
				\left| \sum_{i \not\in \cS} \varepsilon_i^2 \one\left(|\varepsilon_i| > \frac{\hat \s_S}{\s}\s_{i}\tau \right) - \sum_{i \not\in \cS} \varepsilon_i^2 \one\left(|\varepsilon_i| > \s_{i}\tau \right) \right|,
			\end{equation}
			where $\tau=3\sqrt{\log (1 + \|\bs{\S}\|_F^2 / s^2)}$.
			To this end, we use the following two lemmas. The first lemma provides a bound for the expression in \eqref{eq:eps-hat-sigma} in probability. The second lemma establishes a bound on its mean squared error under the additional condition that $\bs{\S}$ is positive definite.  The proofs of the lemmas will be given at the end of this subsection.
			
			\begin{lemma}\label{lem:unsconf}
				There exist absolute constants $c>0$ and $C>0$ such that,
				for any $\eta \in \big(c{(s+\|\bs{\S}\|_F)}/{d},1\big)$,
				\begin{align}
					& \prob\left(\left|\sum_{i \not\in \cS} \varepsilon_i^2  \left(\one\left(|\varepsilon_i| > \frac{\hat{\s}_S}{\s} \s_{i}   \tau \right) - \one\left(|\varepsilon_i| > \s_{i} \tau \right)\right)\right| \le C \frac{s+\|\bs{\S}\|_F}{\eta^2} \tau^3 e^{-\tau^2/4}\right)
					\\
					& \qquad \ge 1 - \eta.
				\end{align}
			\end{lemma}

			\begin{lemma}
				\label{lem:squad}
				If $\bs{\S}$ is positive definite there exists an absolute constant $c>0$ such that,
				for any $\tau \ge 1$,
				\begin{align}
					&\esp \left[\varepsilon_i^2 \abs{ \one\pth{|\varepsilon_i|>\frac{\hat \s_S}{\s}\s_i\tau }  - \one(|\varepsilon_i|>\s_i \tau) }\right]\\
					\le~&   c  \tau^2\pth{1+\sqrt{(\bs{\S}^{-1})_{ii}}}\frac{s+\|\bs{\S}\|_F}{d}e^{-\frac{\tau^2}{16}}
					+  c\tau^2\pth{\frac{s+\|\bs{\S}\|_F}{d}}^2.
				\end{align}
			\end{lemma}
			Since $s\le \|\bs{\S}\|_F$ and $\tau=3\sqrt{\log (1 + \|\bs{\S}\|_F^2 / s^2)}$ we obtain that
			\[
			(s+\|\bs{\S}\|_F) \tau^3 e^{-\tau^2/4} \lesssim s \log (1 + \|\bs{\S}\|_F^2 / s^2) =\psi^*(s,\bs{\S}).
			\]
			Applying Markov's inequality to the random variable under the expectation in (\ref{eq:hat-Nstar-sparse-3a}) and using Lemma~\ref{lem:unsconf}   we obtain that 
			for any $\eta> c {\|\bs{\S}\|_F}/{d}$ there exists a constant $c_\eta$ such that
			\begin{equation}
				\label{eq:hat-Nstar-sparse-4}
				\prob\pth{ \left|\sum_{i \not\in \cS} \varepsilon_i^2 \one\left(|\varepsilon_i| > \frac{\hat \s_S}{\s}\s_{i}\tau \right) - \alpha_s \frac{\hat{\s}_S^2}{\s^2}  \sum_{i=1}^d \s_{i}^2\right| \le \frac{c_\eta}{4} \psi^*(s,\bs{\S}) }
				\ge 1-\frac{\eta}{2}.
			\end{equation}
			Applying now Markov's inequality to the random variable under the expectation in~\eqref{eq:hat-Nstar-sparse-2} and invoking~\eqref{eq:hat-Nstar-sparse-1} and~\eqref{eq:hat-Nstar-sparse-4} we complete the proof of \eqref{eq:unknownvar1-sparse1}.

			We now prove~\eqref{eq:unknownvar1-sparse2}. Under the condition that $\bs{\S}$ is positive definite, we apply Lemma~\ref{lem:squad} with $s\le \|\bs{\S}\|_F$ and $\tau=3\sqrt{\log (1 + \|\bs{\S}\|_F^2 / s^2)}$. This yields:
			\begin{align}
				& \esp\left|\sum_{i \not\in \cS} \varepsilon_i^2 \left(\one\left(|\varepsilon_i| > \frac{\hat{\s}_S}{\s} \s_{i} \tau \right) - \one\left(|\varepsilon_i| > \s_{i} \tau\right)\right)\right| \\
				&\quad  \le  
				\sum_{i \not\in \cS} \esp \left(\varepsilon_i^2 \abs{\one\left(|\varepsilon_i| > \frac{\hat{\s}_S}{\s} \s_{i} \tau \right) - \one\left(|\varepsilon_i| > \s_{i} \tau\right) }\right) 
				\\
				& \quad \lesssim \frac{1}{d}\sum_{i=1}^d \left(\sqrt{(\bs{\S}^{-1})_{ii}} + 1\right) \|\bs{\S}\|_F \frac{\log(1 + \|\bs{\S}\|_F^2 / s^2)}{(1 + \|\bs{\S}\|_F^2 / s^2)^{9/16}}  + \frac{\|\bs{\S}\|_F^2}{d} \log(1 + \|\bs{\S}\|_F^2 / s^2)
				\\
				& \quad \lesssim \frac{s}{d}\sum_{i=1}^d \left(\sqrt{(\bs{\S}^{-1})_{ii}} + 1\right)   + \frac{\|\bs{\S}\|_F^2}{d} \log(1 + \|\bs{\S}\|_F^2 / s^2).
			\end{align}
			Notice that $(\bs{\S}^{-1})_{ii}\le \lambda_{\max}(\bs{\S}^{-1})$. Hence,
			\begin{align}
				&\esp\left|\sum_{i \not\in \cS} \varepsilon_i^2 \left(\one\left(|\varepsilon_i| > \frac{\hat{\s}_S}{\s} \s_{i} \tau \right) - \one\left(|\varepsilon_i| > \s_{i} \tau\right)\right)\right| \\
				&\qquad\qquad \lesssim s \left(\sqrt{\lambda_{\max}(\bs{\S}^{-1})} + 1\right)  + \frac{\|\bs{\S}\|_F^2}{d} \log(1 + \|\bs{\S}\|_F^2 / s^2).
			\end{align}
			Combining this bound with \eqref{eq:hat-Nstar-sparse-1}, \eqref{eq:hat-Nstar-sparse-2} and \eqref{eq:hat-Nstar-sparse-3a} concludes the proof of~\eqref{eq:unknownvar1-sparse2}.

			
			\begin{proof}[Proof of Lemma~\ref{lem:unsconf}]
				Given any $\eta$ such that $\frac{c(s+\|\bs{\S}\|_F)}{d} < \eta <1$, define a random event
				\[
				\mathcal{A} := \left\{\frac{|\hat{\s}_S^2-\s^2|}{\s^2} \le \delta := \frac{c(s+\|\bs{\S}\|_F)}{2\eta d}\right\},
				\]
				where $c > 0$ is a constant to be chosen.
				Applying the Markov inequality to Theorem~\ref{thm:hat-sigma-sparse} yields that by choosing $c$ sufficiently large, $\prob(\mathcal{A}^c) \le \eta / 2$.
				Note that $\delta \le \frac{1}{2}$.
				For any $\tau > 0$,
				\begin{align}
					& \esp \left(  \varepsilon_i^2 \left| \one\left(|\varepsilon_i| > \frac{\hat{\s}_S}{\s} \s_{i} \tau \right) - \one\left(|\varepsilon_i| > \s_{i} \tau\right)  \right|  \one(\mathcal{A}) \right) \\
					\le~& \esp \left(  \varepsilon_i^2  \one\left( \tau\sqrt{1-\delta}   \le \frac{|\varepsilon_i|}{\s_{i}} \le \tau \sqrt{1+\delta}  \right)  \one(\mathcal{A}) \right) \\
					\le~& \s_i^2\esp \left(  \frac{\varepsilon_i^2}{\s_i^2}  \one\left( \tau \sqrt{1-\delta}   \le \frac{|\varepsilon_i|}{\s_{i}} \le \tau \sqrt{1+\delta}  \right)   \right) \\
					\lesssim~& \s_i^2 \delta \tau^3 e^{-\tau^2/4},
				\end{align}
				where in the last step we have used the fact that, for $b>a>0$,
				\[
				\int_a^b x^2 e^{-x^2/2}dx \le  b^2 e^{-a^2/2} (b-a).
				\]
				Consequently, since $\s_i^2\le 1$ by Assumption~\ref{as:var0}, we have
				\[
				\esp\left(\Big|\sum_{i \not\in \cS} \varepsilon_i^2  \left(\one\left(|\varepsilon_i| > \frac{\hat{\s}_S}{\s} \s_{i}   \tau \right) - \one\left(|\varepsilon_i| > \s_{i} \tau \right)\right)\Big|\one(\mathcal{A})\right)
				\lesssim \frac{s+\|\bs{\S}\|_F}{\eta } \tau^3 e^{-\tau^2/4}.
				\]
				Applying the Markov inequality yields that
				\[
				\prob\qth{\sth{ \left|\sum_{i \not\in \cS} \varepsilon_i^2  \left(\one\left(|\varepsilon_i| > \frac{\hat{\s}_S}{\s} \s_{i}   \tau \right) - \one\left(|\varepsilon_i| > \s_{i} \tau \right)\right)\right|\ge \frac{C(s+\|\bs{\S}\|_F)}{\eta^2 } \tau^3 e^{-\tau^2/4} }
					\cap \mathcal{A}}
				\le \eta/2,
				\]
				where $C$ is a constant.
				In light of the above and that $\prob(\mathcal{A}^c) \le \eta / 2$, we obtain the desired result.
			\end{proof}
			
			
			The rest of this subsection is devoted to the proof of Lemma \ref{lem:squad}, where
			$\bs{\S}$ is assumed positive definite. The idea of the proof is to use a leave-one-out argument.
			For $i\in[d]$, set $\bs{\varepsilon}_{-i}=(\varepsilon_1,\dots,\varepsilon_{i-1},\varepsilon_{i+1},\dots,\varepsilon_{d})$ and $\bs{Y_{-i}}=(Y_1,\dots,Y_{i-1},Y_{i+1},\dots,Y_d)$.
			Since only one observation is dropped, the corresponding variance estimator using $\bs{Y_{-i}}$ is close to the original $\hat \s_S^2$ with high probability.
			On the other hand, since $\bs{\S}$ is positive definite, all its principal sub-matrices are positive definite, the conditional distribution of $\varepsilon_i$ given $\bs{\varepsilon}_{-i}$ is non-degenerate and has the form (see, e.g.,~\citep{E83})
			\begin{equation}
				\label{eq:Gauss-cond}
				\varepsilon_i|\bs{\varepsilon}_{-i} \sim \calN(\cov(\varepsilon_i, \bs{\varepsilon}_{-i})\cov(\bs{\varepsilon}_{-i})^{-1}\bs{\varepsilon}_{-i}, \s_i^2-\cov(\varepsilon_i, \bs{\varepsilon}_{-i})\cov(\bs{\varepsilon}_{-i})^{-1}\cov(\bs{\varepsilon}_{-i},\varepsilon_i)),
			\end{equation}
			where, by the Schur complement (see, e.g., \cite[(0.7.3.1)]{HJ12}),
			\begin{equation}
				\label{eq:Schur}
				\s_i^2-\cov(\varepsilon_i, \bs{\varepsilon}_{-i})\cov(\bs{\varepsilon}_{-i})^{-1}\cov(\bs{\varepsilon}_{-i},\varepsilon_i) = \frac{1}{(\bs{\S}^{-1})_{ii}} \in [0,\s_i^2].
			\end{equation}
			We will use these remarks to derive the following preliminary lemma.
			
			\begin{lemma}\label{lem:squad1}
				Let
				$t\asymp \sigma^2$.
				Then for any $\tau > 0$,
				\[
				\prob\qth{\frac{\hat{\s}_{S}(t)}{\s} \wedge 1 \le \frac{|\varepsilon_i|}{\s_i \tau}\le  \frac{\hat{\s}_{S}(t)}{\s} \vee 1}
				\lesssim \sqrt{(\bs{\S}^{-1})_{ii}}\frac{s+\|\bs{\S}\|_F}{d}e^{-\frac{\tau^2}{16}} + \pth{\frac{s+\|\bs{\S}\|_F}{d}}^2.
				\]
			\end{lemma}
			\begin{proof}
				Let $\hat F_{-i}(t)$ denote the empirical cumulative distribution function of $\{Y_j^2 / \s_{j}^2:j\ne i\}$.
				Similar to the estimator $\hat\s_S^2$ in \eqref{eq:hat-sigma-t-sparse}, define the corresponding estimator and its relative error as
				\[
				\hat{\s}_{-i}^2(t) := \frac{t}{F^{-1}(\hat F_{-i}(t))},
				\qquad
				\d_{-i}(t):= \abs{\frac{\hat{\s}_{-i}^2(t)}{\s^2}-1},
				\]
				respectively.
				Let $E_1:=\{ \d_{-i}(t) \le 1/100\}$.
				It follows from \Cref{lmm:sigma-t-prob} that $\prob[E_1^c]\lesssim (\frac{s+\|\bs{\S}\|_F}{d})^2$ since $t\asymp \sigma^2$.
				We get
				\begin{equation}
					\label{eq:squad-ub-1}
                    \begin{aligned}
					& \prob\qth{ \frac{\hat{\s}_{S}(t)}{\s} \wedge 1 \le \frac{|\varepsilon_i|}{\s_i \tau}\le  \frac{\hat{\s}_{S}(t)}{\s} \vee 1  } \\
					&\qquad \lesssim  \prob\qth{ \frac{\hat{\s}_{S}(t)}{\s} \wedge 1 \le \frac{|\varepsilon_i|}{\s_i \tau}\le  \frac{\hat{\s}_{S}(t)}{\s} \vee 1 , E  } + \pth{\frac{s+\|\bs{\S}\|_F}{d}}^2.
                    \end{aligned}
				\end{equation}
				Additionally, we get under event $E_1$ that $\hat F_{-i}(t)$ is strictly bounded away from zero and from one.
				Using the fact that $|\hat F_{-i}(t)-\hat F(t)|\le \frac{1}{d}$ and the mean value theorem yields that, under the event~$E_1$,
				\[
				\abs{\frac{\hat{\s}_{-i}^2(t)}{\s^2} - \frac{\hat{\s}_{S}^2(t)}{\s^2}}
				\le c/d,
				\]
				for an absolute constant $c>0$. Without loss of generality, we assume that $d$ is sufficiently large such that $c / d \le 0.01$.
				Then,
				\begin{align}
					&\prob\qth{ \frac{\hat{\s}_{S}^2(t)}{\s^2} \wedge 1 \le \frac{\varepsilon_i^2}{\s_i^2 \tau^2}\le  \frac{\hat{\s}_{S}^2(t)}{\s^2} \vee 1 , E_1}\\
					\le~& \prob\qth{ \pth{\frac{\hat{\s}_{-i}^2(t)}{\s^2} - \frac{c}{d}} \wedge 1 \le \frac{\varepsilon_i^2}{\s_i^2 \tau^2}\le \pth{\frac{\hat{\s}_{-i}^2(t)}{\s^2} + \frac{c}{d}} \vee 1  , E_1} \\
					\le~& \prob\qth{ 1 - \delta_{-i}(t) - \frac{c}{d} \le \frac{\varepsilon_i^2}{\s_i^2 \tau^2}\le 1 + \delta_{-i}(t) + \frac{c}{d} , E_1}. \label{eq:squad-ub-2}
				\end{align}
				Since $\delta_{-i}(t)$ is measurable with respect to $\bs{\varepsilon}_{-i}$ we can use the conditional distribution \eqref{eq:Gauss-cond} to obtain an upper bound on the expression in~\eqref{eq:squad-ub-2}.
				Specifically, the conditional distribution of $\frac{\varepsilon_i}{\s_{i}}$ given $\bs{\varepsilon}_{-i}$ is $\calN(\mu_i(\beps_{-i}),\i_i^2)$, where
				\[
				\i_i^2=\frac{1}{\s_i^2 (\bs{\S}^{-1})_{ii}}\in [0,1],\quad \mu_i(\beps_{-i})\sim \calN(0,1-\i_i^2).
				\]
				Let $l:= \tau \sqrt{1-\d_{-i}(t)-c/d}$, $u:=\tau\sqrt{1+\d_{-i}(t)+c/d}$.
				The length of the interval $[l, u]$ is at most $L=2(\d_{-i}(t)+c/d)$.
				Then we get
				\begin{align}
					\prob\qth{|\varepsilon_i| / \s_{i}\in [l, u], E_1 }
					\le~& \esp\qth{\one(E_1) \sqrt{\frac{2}{\pi}}\frac{L}{\i_i} \exp\pth{- \frac{(l - |\mu_i(\beps_{-i})|)_+^2}{2\i_i^2}  }  },\label{eq:squad-ub-3}
				\end{align}
				where the last inequality uses the fact that, for $0\le a\le b$ and $X\sim \calN(\mu, \s^2)$,
				\begin{align}
					\prob\qth{ |X| \in [a,b] }
					& = \int \frac{1}{\sqrt{2\pi} \s}e^{-\frac{(x-\mu)^2}{2\s^2}}\one(a\le |x|\le b) d x \\
					& \le 2 \frac{b-a}{\sqrt{2\pi} \s} \exp\pth{-\frac{1}{2\s^2} \pth{\min_{x:a\le|x|\le b}|x-\mu| }^2 }\\
					& \le \sqrt{\frac{2}{\pi}} \frac{b-a}{\s}\exp\pth{-\frac{(a-|\mu|)_+^2}{2\s^2}}.
				\end{align}
				It remains to bound from above the right-hand side of~\eqref{eq:squad-ub-3}.
				Define $E_2=\{ |\mu_i(\beps_{-i})| \le \tau/2 \}$. Then,
				$$\prob[E_2^c] \le \exp\pth{-\frac{(\tau/2)^2}{2(1-\i_i^2)}  } \le e^{-\tau^2/8}.$$
				Under the events $E_1$ and $E_2$ we have $l - |\mu_i(\beps_{-i})| \ge 0.48 \tau$.
				Therefore,
				\begin{align}
					& \phantom{{}={}} \esp\qth{\one(E_1) L \exp\pth{- \frac{(l - |\mu_i(\beps_{-i})|)_+^2}{2\i_i^2}  }  } \\
					& \le \esp\qth{\one(E_1) \one(E_2) L \exp\pth{- \frac{(l - |\mu_i(\beps_{-i})|)_+^2}{2}  }  }
					+ \esp\qth{\one(E_1) \one(E_2^c) L   } \\
					& \le \esp\qth{  L \one(E_1) } \exp\pth{ -\tau^2/9 }
					+ \sqrt{\prob[E_2^c] \esp\qth{ L^2 \one(E_1) }}.
				\end{align}
				Recall that $L\le \delta_{-i}(t) + \frac{c}{d}$.
				By an  argument similar to \eqref{eq:ub-mse} we obtain that
				\[
				\esp\qth{ L^2 \one(E_1) }
				\lesssim \pth{\frac{s+\|\bs{\S}\|_F}{d}}^2.
				\]
				Using the above inequalities in~\eqref{eq:squad-ub-3} we get the following upper bound on the probability in~\eqref{eq:squad-ub-2}:
				\[
				\prob\qth{ \frac{\hat{\s}_{S}^2(t)}{\s^2} \wedge 1 \le \frac{\varepsilon_i^2}{\s_i^2 \tau^2}\le  \frac{\hat{\s}_{S}^2(t)}{\s^2} \vee 1 , E_1}
				\lesssim \s_i \sqrt{(\bs{\S}^{-1})_{ii}}\frac{s+\|\bs{\S}\|_F}{d}\exp\pth{-\frac{\tau^2}{16}}.
				\]
				The conclusion follows from this inequality and~\eqref{eq:squad-ub-1}.
			\end{proof}

			Equipped with Lemma~\ref{lem:squad1}, we are now in a position to prove Lemma~\ref{lem:squad}.
			\begin{proof}[Proof of Lemma~\ref{lem:squad}]
				Define the following two events:
				\[
				{\rm E}_1 = \{\hat t\in \calT\}, \quad  {\rm E}_2 = \sth{\hat\s_S^2\le \frac{3}{2}\s^2},
				\]
				where $\prob[ {\rm E}_1^c] \lesssim {\|\bs{\S}\|_F^2}/{d^2}$ and $ \prob[ {\rm E}_2^c] \lesssim {\|\bs{\S}\|_F^2}/{d^2}$ according to Lemmas~\ref{lmm:hat-t-property} and~\ref{lmm:hat-sigma-whp}, respectively.
				Set $ {\rm E}= {\rm E}_1\cap  {\rm E}_2$.
				We have
				\begin{align}
					& \esp \left(\varepsilon_i^2 \abs{ \one\pth{|\varepsilon_i|>\frac{\hat \s_S}{\s}\s_i\tau }  - \one(|\varepsilon_i|>\s_i \tau) }\right) \le
					\esp \left( \varepsilon_i^2  \one\pth{ \frac{\hat \s_S}{\s}\wedge 1 \le \frac{|\varepsilon_i|}{\s_i\tau} \le \frac{\hat \s_S}{\s} \vee 1}\right)\\
					&\le
					\frac{3}{2}\s_i^2\tau^2 \esp \left(\one\pth{ \frac{\hat \s_S}{\s}\wedge 1 \le \frac{|\varepsilon_i|}{\s_i\tau} \le \frac{\hat \s_S}{\s} \vee 1}\one( {\rm E})\right)
					\\
					&\quad + \esp  \left(\varepsilon_i^2  \one\pth{ \frac{\hat \s_S}{\s}\wedge 1 \le \frac{|\varepsilon_i|}{\s_i\tau} \le \frac{\hat \s_S}{\s} \vee 1} \one( {\rm E}^c)\right).
					\label{eq:invertible-mse-1}
				\end{align}
				Applying Lemma~\ref{lem:squad1} we get 
				\begin{align}
					& \s_i^2\tau^2\esp \left(\one\pth{ \frac{\hat \s_S}{\s}\wedge 1 \le \frac{|\varepsilon_i|}{\s_i\tau} \le \frac{\hat \s_S}{\s} \vee 1}\one({\rm E})\right) \\
					\le~& 
					\s_i^2\tau^2\sum_{t\in \calT}\prob\qth{ \frac{\hat \s_S}{\s}\wedge 1 \le \frac{|\varepsilon_i|}{\s_i\tau} \le \frac{\hat \s_S}{\s} \vee 1 }
					\\
					\lesssim~& \s_i^2\tau^2 \pth{\sqrt{(\bs{\S}^{-1})_{ii}}\frac{s+\|\bs{\S}\|_F}{d}e^{-\frac{\tau^2}{16}} + \pth{\frac{s+\|\bs{\S}\|_F}{d}}^2}. \label{eq:invertible-mse-2}
				\end{align}
				Next we bound from above the second term in \eqref{eq:invertible-mse-1}, which can be further decomposed into two terms using the fact that ${\rm E}^c= ({\rm E}_2\cap {\rm E}_1^c)\cup {\rm E}_2^c$.
				For the term corresponding to the event ${\rm E}_2\cap {\rm E}_1^c$ we have
				\begin{equation}
					\label{eq:invertible-mse-3}
					\esp  \left(\varepsilon_i^2  \one\pth{ \frac{\hat \s_S}{\s}\wedge 1 \le \frac{|\varepsilon_i|}{\s_i\tau} \le \frac{\hat \s_S}{\s} \vee 1} \one({\rm E}_1^c \cap {\rm E}_2)\right)
					\le 
					\frac{3}{2}\s_i^2\tau^2 \prob[{\rm E}_1^c]
					\lesssim  \s_i^2\tau^2 \frac{\|\bs{\S}\|_F^2}{d^2}.
				\end{equation}
				Next, for the term corresponding to the event ${\rm E}_2^c$ we have
				\begin{align}
					&\esp  \left(\varepsilon_i^2  \one\pth{ \frac{\hat \s_S}{\s}\wedge 1 \le \frac{|\varepsilon_i|}{\s_i\tau} \le \frac{\hat \s_S}{\s} \vee 1} \one({\rm E}_2^c)\right)
					\le \esp \left(\varepsilon_i^2  \one(|\varepsilon_i| \ge \s_i \tau )\one({\rm E}_2^c)\right) 
					\\
					&\qquad \le \sqrt{{\mathbf E}  \left(\varepsilon_i^4  \one(|\varepsilon_i| \ge \s_i \tau )\right)} \frac{\|\bs{\S}\|_F}{d}
					\lesssim \s_i^2 \sqrt{(\tau^3 \vee 1 ) e^{-\tau^2/2}} \frac{\|\bs{\S}\|_F}{d}, \label{eq:invertible-mse-4}
				\end{align}
				where we have used first the Cauchy-Schwarz inequality and then the fact that
				\[
				\int_{\tau}^{\infty } x^4 e^{-x^2/2}dx
				= e^{-\tau^2/2} (\tau^3+3\tau) + 3 \int_{\tau}^{\infty} e^{-x^2/2}dx.
				\]
				Plugging \eqref{eq:invertible-mse-2}, \eqref{eq:invertible-mse-3}, and \eqref{eq:invertible-mse-4} into \eqref{eq:invertible-mse-1} we complete the proof.
			\end{proof}
			
				\subsubsection{Proof of Theorem~\ref{thm:unknownvar2}}
				
				We will use the following lemma about the deviations of $\hat{\s}^2_\eta$, the proof of which will be given at the end of this subsection.
				
				\begin{lemma}\label{lem:etasigma}
					Let $\eta \in (0, 1)$. If $s \le d / 100$ then
					\[
					\prob(\hat{\s}^2_\eta / \s^2 \in (q_{\eta / 20} / q_{1 - \eta / 20}, 1)) \ge 1 - \eta / 4.
					\]
				\end{lemma}
				On the event defined in Lemma~\ref{lem:etasigma}, we have that, almost surely, $|\max\{\hat{\s}_S, \hat{\s}_\eta\} - \s| \le |\hat{\s}_S - \s|$. Using this and following the same argument as in the proof of Theorem~\ref{thm:unknownvar1}, we only need to show that there exists some constant $c_\eta>0$ depending on $\eta \in (0, 1)$ such that
				\begin{equation}\label{eq:esteta1}
					\prob\left(
					\left| \sum_{i \not\in \cS} \varepsilon_i^2 \one\left(|\varepsilon_i| > \frac{\max \{\hat \s_S, \hat{\s}_\eta\}}{\s}\s_{i}\tau_\eta \right) - \sum_{i \not\in \cS} \varepsilon_i^2 \one\left(|\varepsilon_i| > \s_{i}\tau_\eta \right) \right| \ge c_\eta \psi^*(s, \bs{\S}) \right) \ge 1 - \eta / 2.
				\end{equation}
				To achieve this goal, we consider the event 
				\[
				\mathcal{A}(\eta) := \left\{\frac{|\hat{\s}^2_S - \s^2|}{\s^2} \le \delta := c\frac{\|\bs{\S}\|_{F}}{\eta d}, \hat{\s}^2_\eta / \s^2 \in (q_{\eta / 20} / q_{1 - \eta / 20}, 1)\right\}.
				\]
				Using Lemma~\ref{lem:etasigma} and Theorem~\ref{thm:hat-sigma-sparse} we obtain that, for a constant $c > 0$ large enough, $\prob(\mathcal{A}(\eta)) \ge 1 - \frac{3 \eta}{8}$. Arguing as in the proof of Lemma~\ref{lem:unsconf} we get  
				\begin{align}
					& \esp \left(  \varepsilon_i^2 \left| 
					\one\left(|\varepsilon_i| > \frac{\max\{\hat{\s}_S, \hat{\s}_\eta\}}{\s} \s_{i} \tau_\eta \right) - 
					\one\left(|\varepsilon_i| > \s_{i} \tau_\eta\right)  \right|  \one(\mathcal{A}(\eta)) \right) \\
					\le~& \esp \left(  \varepsilon_i^2  \one\left( \max\left\{\tau_\eta\sqrt{1-\min(1,\delta)}, \tau_\eta \sqrt{\frac{q_{\eta / 20}}{q_{1 - \eta / 20}}}\right\}   \le \frac{|\varepsilon_i|}{\s_{i}} \le \tau_\eta \sqrt{1+\delta}  \right)  \one(\mathcal{A}(\eta)) \right) \\
					\le~& \esp \left(  \varepsilon_i^2  \one\left( \max\left\{\tau_\eta\sqrt{1-\min(1,\delta)}, \tau \right\}   \le \frac{|\varepsilon_i|}{\s_{i}} \le \tau_\eta \sqrt{1+\delta}  \right) \one(\mathcal{A}(\eta)) \right) \\
					\lesssim & \s_i^2 \delta (1 + \delta)\tau_\eta^3 e^{-\tau^2/2}.
				\end{align}
				Recalling that $\tau = 3 \sqrt{\log (1 + \|\bs{\S}\|_F^2 / s^2)}$ 
				and acting in the same way as in Lemma~\ref{lem:unsconf} we obtain~\eqref{eq:esteta1} for any $\eta \in (0, 1)$.
				
				\begin{proof}[Proof of Lemma~\ref{lem:etasigma}]
					Without loss of generality, we assume that all $\s_i$'s are equal to $1$. Then we have 
					\begin{align}
						\prob(\hat{M} \ge \s^2 q_{1 - \eta / 20}) & \le \prob\left(\hat{F}_{Y^2}\left(\s^2 q_{1 - \eta / 20}\right) \le \frac{1}{2}\right) \le \prob\left(\hat{F}_{\varepsilon^2}\left(q_{1 - \eta / 20}\right) \le 0.51\right) \\
						& = \prob\left(1 - \hat{F}_{\varepsilon^2}\left(q_{1 - \eta / 20}\right) \ge 0.49\right) \le \eta / 8,
					\end{align}
					where the last bound uses Markov's inequality. Similarly, we have
					\begin{align}
						\prob(\hat{M} \le \s^2 q_{\eta / 20}) & \le \prob\left(\hat{F}_{Y^2}\left(\s^2 q_{\eta / 20}\right) \ge \frac{1}{2}\right) \le \prob\left(\hat{F}_{\varepsilon^2}\left(q_{\eta / 20}\right) \ge 0.49\right)\le \eta / 8.
					\end{align}
					The desired result follows from the above inequalities and the union bound.
				\end{proof}
			
			\subsubsection{Proof of Theorem~\ref{thm:unknownlb}}
			
			We only prove the lower bound of Theorem~\ref{thm:unknownlb}. The proof of the upper bounds (items (i) and (ii) of Theorem~\ref{thm:unknownlb}) follows from arguments quite analogous to the proofs of Theorems~\ref{thm:unknownvar1} and~\ref{thm:unknownvar2}. It is therefore omitted.
			
			Notice that for $s\le \rho$ the rates stated in the lower bounds of Theorems~\ref{thm:lower-known-sigma-rho} and~\ref{thm:unknownlb} are the same while the result of Theorem~\ref{thm:lower-known-sigma-rho} is stronger since it holds for a smaller minimax risk. This proves the lower bound of Theorem~\ref{thm:unknownlb} for $s\le \rho$.  
			
			It remains only to prove the lower bound for the case $s > \rho$. 
			Similarly to the proof of Theorem~\ref{thm:lower-known-sigma-rho}, we proceed by a reduction to the lower bound for the model with independent observations. Without causing confusion, we will use the same notation as in the proof of Theorem~\ref{thm:lower-known-sigma-rho} for somewhat different but analogous quantities. Given a fixed loss function $w(\cdot)$ satisfying the assumptions of Theorem~\ref{thm:unknownlb} and any $\phi>0$, we now define $R^*(s, d, \rho, \phi)$ as the  minimax risk 
			\[
			R^*(s, d, \rho, \phi) := \inf_{\hat{T}}\sup_{(\bs{\t}, \bs{\S}) \in \T(s, d, \rho)} \sup_{\s > 0} \esp\left[w\left(\s^{-1}\phi^{-1/2} |\hat{T} - \|\bs{\t}\|_2|\right)\right].
			\]
			With this notation, our aim is to prove that $R^*(s, d, \rho, \phi^*(s,\rho^2))\gtrsim 1$ for $s > \rho$. Moreover, arguing analogously to the proof of Theorem~\ref{thm:lower-known-sigma-rho} we see that it suffices to obtain this result for $\rho \ge  2 \sqrt{d}$. We also assume without loss of generality that $s \ge 2$.
			

		
		For a positive integer $r$, we define $p=\lfloor d/r \rfloor$. Let $\bs{\S}_0^*$ be the block diagonal covariance matrix defined as in \eqref{eq:sigma0} with the only difference that we replace $\bs{I}_{(d - rp) \times (d - rp)}$ by $\bs{1}_{(d - rp) \times (d - rp)}$, that is, by the $(d - rp) \times (d - rp)$ dimensional matrix with all entries equal to $1$.
		For a positive integer $k\le p$, we define a class of vectors $\bs{\t}\in \RR^d $ as follows:
		\[
		\bar{\T} := \left\{(\underset{r \;\mathrm{entries}}{\underbrace{\tilde{\t}_1, \cdots, \tilde{\t}_1}}, \cdots, \underset{r \;\mathrm{entries}}{\underbrace{\tilde{\t}_p, \cdots, \tilde{\t}_p}}, 0, \cdots, 0)^\top \in \RR^d:  \sum_{i=1}^p \one(\tilde{\t}_i \neq 0) \le k  \right\}.
		\]
		By the choice of $p$ we have $d-rp< r$. If the triplet $(r, p, k)$ is such that
		\begin{equation}\label{eq:restr3}
			k \le p,\quad\quad rk \le s, \quad \text{and} \quad r^2 (p + 1) \le \rho^2,
		\end{equation}
		then for any $\bs{\t} \in \bar{\T}$, we have $(\bs{\t}, \bs{\S}_0^*) \in \T(s, d, \rho)$. 
		Moreover, for $\bs{\t}\in\bar{\T}$, $\bs{\S}=\bs{\S}_0^*$, model \eqref{eq:model} is equivalent to the model with $p+1$ independent observations $(\tilde{Y}_1, \dots, \tilde{Y}_{p + 1})$ such that
		$\tilde{Y}_i \sim \cN(\tilde{\t}_i, \s^2)$, $i\in[p]$, where at most $k$ of $\tilde{\t}_i$'s are nonzero, and $\tilde{Y}_{p + 1} \sim \cN(0, \s^2)$. 
		We denote by ${\mathbf E}_{\tilde{\bs{\t}},\sigma}^*$, where 
		$\tilde{\bs{\t}}=(\tilde 
		\t_1,\dots,\tilde \t_p)$, the expectation with respect to the joint distribution of $(\tilde{Y}_1, \ldots, \tilde{Y}_p, \tilde{Y}_{p + 1})$ and we define, for any $\phi>0$,
		\begin{equation}\label{eq:raugmented}
			\tilde R^*(k, p, \phi) := \inf_{\hat{T}}\sup_{\tilde{\bs{\t}} \in B_0(k, p) 
			} \sup_{\s > 0} {\mathbf E}_{\tilde{\bs{\t}},\sigma}^*\left[w\left(\s^{-1} \phi^{-1/2}|\hat{T} - \|\tilde{\bs{\t}}\|_2|\right)\right].
		\end{equation}
		Since $\|\bs{\t}\|_2 = \sqrt{r} \sqrt{\sum_{i=1}^p \tilde{\t}_i^2}$ we obtain following the same lines as in \eqref{eq:r-to-r} that, for any $\phi>0$,
		\begin{equation}\label{eq:rtor}
			R^*(s, d, \rho, \phi) \ge \tilde R^*(k, p, r^{-1}\phi).
		\end{equation}
		In what follows, we set $r=\lfloor \frac{\rho^2}{2d} \rfloor$ and $k=\lfloor \frac{s}{r} \rfloor$. Then the triplet $(r,p=\lfloor d/r \rfloor,k)$ satisfies the inequalities in \eqref{eq:restr3}. Since $\rho\ge 2\sqrt{d}$ and $s>\rho>r$ we have that $r$ and $k$ are positive integers, so the above triplet is a feasible choice. Moreover, with such a choice, we have that $r\le \frac{\rho^2}{2d}<\frac{s}{2}\le \frac{d}{2}$ as $\rho<s\le d$. Therefore,
		\begin{align}
			& k\le \frac{s}{r}\le \frac{4sd}{\rho^2},\\
			& k\ge \frac{s}{r}-1 \implies r k \ge s - r \ge s / 2,\\
			& p\ge \frac{d}{r}-1 \implies rp\ge d-r\ge \frac{d}{2} \implies p\ge \frac{d}{2r}\ge \frac{d^2}{\rho^2}.
		\end{align}
		Next, recall that $s\le \alpha d$ for $\alpha>0$ small enough by the assumption of the theorem. Thus, in the following we assume without loss of generality that $\alpha\le 1/2$, so that $\rho\le d/2$. It follows that 
		\begin{align}
			\frac{k^2}{p}&\ge \frac{(s/r-1)^2}{(d/r)} = \frac{(s-r)^2}{dr}
			\\
			&\ge \frac{(2s\rho -\rho^2)^2}{2d^2\rho^2} \ge \frac{(2d -\rho)^2}{2d^2} \ge 1.
		\end{align}
		Consequently,
		\[
		r \phi^*(k, p) = \frac{r k}{1 \vee \log(k / \sqrt{p})} \gtrsim \frac{s}{1 \vee \log(s / \rho)} = \phi^*(s, \rho^2).
		\]
		Note also that, by choosing $\alpha$ small enough the condition $p\ge \bar C$ of Lemma~\ref{lem:unknown-indep} below is met. Indeed, $p\ge d/(2r)\ge d/s\ge 1/\alpha$.
		Using these remarks, \eqref{eq:rtor} and Lemma~\ref{lem:unknown-indep} we obtain that, for some absolute constant $c > 0$, 
		\[
		R^*(s, d, \rho, \phi^*(s, \rho^2)) \ge R^*(s, d, \rho, cr \phi^*(k, p)) \ge \tilde R^*(k, p, c\phi^*(k, p)) \gtrsim 1,
		\]
		This concludes the proof.

		\begin{lemma}\label{lem:unknown-indep}
			Let $\tilde{R}^*(k, p, \phi)$ be defined in \eqref{eq:raugmented}, where $w$ is a loss function satisfying the assumptions of Theorem~\ref{thm:unknownlb}. Then there exists an absolute constant $\bar C>0$ such that for any $p\ge \bar C$, any $k \ge \sqrt{p}$, and any absolute constant $c>0$ we have $\tilde{R}^*(k, p, c\phi^*(k, p)) \gtrsim 1$.
		\end{lemma}
		
		\begin{proof}
			It suffices to prove the lemma for $c=1$ since the function $w(\cdot/c^{1/2})$ satisfies the assumptions of Theorem~\ref{thm:unknownlb}. The proof follows the argument in \cite[Proposition~3(ii)]{comminges2021adaptive}, with the only difference is that, in our notation, \cite[Proposition~3(ii)]{comminges2021adaptive} proves the result when the sample is $(\tilde{Y}_1, \dots, \tilde{Y}_{p})$, while in our case the sample contains the additional observation $\tilde{Y}_{p + 1}$. Therefore, we only focus on the way to handle this difference and do not reproduce the other details of the proof. 
			
			Given a $p$-dimensional probability distribution $\bs{P}$ and $v>0$, we denote by $\bs{P} \otimes \cN(0, v)$ the distribution of a $(p+1)$-dimensional random vector $(X_1, \ldots, X_{p + 1})$ such that the first $p$ coordinates $(X_1, \ldots, X_p)$ are distributed as $\bs{P}$ and independent from the last coordinate $X_{p+1} \sim \cN(0, v)$.
			
			Let the mixture probability measures $\tilde{\mathbf{P}}_1, \tilde{\mathbf{P}}_2$ and the constant %
			$\tilde\varphi>0$ be defined in the same way as $\mathbf{P}_1$, $\mathbf{P}_2$ and $\varphi$ in the proof of \cite[Proposition~3(ii)]{comminges2021adaptive} with the only difference that $s$ and $d$
			are replaced by $k$ and $p$, respectively. 
			We now follow the same argument as in the proof of \cite[Proposition~3(ii)]{comminges2021adaptive} but considering the distributions $\tilde{\mathbf{P}}_1 \otimes \cN(0, 1)$ and $\tilde{\mathbf{P}}_2 \otimes \cN(0, 1 + \tilde\varphi)$ instead of $\mathbf{P}_1$ and $\mathbf{P}_2$, respectively. It follows that, in order to prove the lemma, it suffices to show that  
			\begin{equation}\label{eq:TVv}
				V(\tilde{\mathbf{P}}_1 \otimes \cN(0, 1), \tilde{\mathbf{P}}_2 \otimes \cN(0, 1 + \tilde\varphi)) \le \frac{5}{8},
			\end{equation}
			where $V(\cdot, \cdot)$ denotes the total variation distance. All other elements of the proof remain the same as in~\cite[Proposition~3(ii)]{comminges2021adaptive}.  Standard properties of the total variation distance yield that
			\begin{align}
				& V(\tilde{\mathbf{P}}_1 \otimes \cN(0, 1), \tilde{\mathbf{P}}_2 \otimes \cN(0, 1 + \tilde\varphi)) \\
				& \qquad \le V(\tilde{\mathbf{P}}_1 \otimes \cN(0, 1), \tilde{\mathbf{P}}_2 \otimes \cN(0, 1)) + V(\tilde{\mathbf{P}}_2 \otimes \cN(0, 1), \tilde{\mathbf{P}}_2 \otimes \cN(0, 1 + \tilde\varphi)) \\
				& \qquad = V(\tilde{\mathbf{P}}_1, \tilde{\mathbf{P}}_2) + V(\cN(0, 1), \cN(0, 1 + \tilde\varphi)).
			\end{align}
			From the proof of \cite[Proposition~3(ii)]{comminges2021adaptive} we have $V(\mathbf{P}_1, \mathbf{P}_2) \le {1}/{2}$, so that $V(\tilde{\mathbf{P}}_1, \tilde{\mathbf{P}}_2) \le {1}/{2}$. Moreover, $V(\cN(0, 1), \cN(0, 1 + \tilde\varphi)) \le c' \tilde\varphi$ for some absolute constant $c' > 0$. Recalling the definition of $\varphi$ in the proof of \cite[Proposition~3(ii)]{comminges2021adaptive} we have that $\tilde\varphi \le \frac{c_0k}{2 p \log(e k^2 / p)}\le c_0/2$, where the constant $c_0>0$ can be chosen as small as needed. By choosing $c_0 \le \frac{1}{4c'}$ we obtain \eqref{eq:TVv}.
		\end{proof}
		
		\subsection{Proof of~\Cref{prop: lower-testing}}
		
			As in the proof of \Cref{thm:lower-known-sigma-rho} we consider two cases. In the first case corresponding to $\sqrt{d} \le \rho \le 2 \sqrt{d}$ we have $\phi(s, \rho^2) \le 2 \phi(s, d)$. Then the result follows by reduction to the model with $\bs{\Sigma} = \bs{I}$ {\rev by noticing that}
			\begin{align}\label{eq:LB-test}
				{\mathcal R}^*(s,d,\rho,\bar r_{\gamma}) \ge   \inf_{\Psi}\Big[\mathbf{P}_{\mathbf{0},\bs{I}} (\Psi=1)  + \sup_{\bs{\t} \in B_0(s, d): \|\bs{\t}\|_2\ge 4 \gamma\s\sqrt{\phi(s, d)}} \mathbf{P}_{\bs{\t},\bs{I}} (\Psi=0)   \Big]\ge {\rev c(\gamma)}
			\end{align}
			{\rev for all $\gamma>0$, where $c(\gamma)>0$ depends only on $\gamma$. This is done following the same lines as in the proof of \cite[Theorem 12]{CCT2017}.
				Note that \cite[Theorem 12]{CCT2017} obtains \eqref{eq:LB-test} for sufficiently small $\gamma$, while we need it here for all $\gamma>0$. Extending the proof in \cite{CCT2017} to all $\gamma>0$ is straightforward by using the inequality  $V(P,Q)\le 1- \exp(-\chi^2(P,Q))/2$ instead of $V(P,Q)\le \sqrt{\chi^2(P,Q)}$, where $V(P,Q)$ and $\chi^2(P,Q)$ are the total variation distance and the chi-square divergence  between two probability measures $P$ and $Q$ (cf. \cite[Theorem 2.2(iii)]{tsybakov2009introduction}). Following the lines of \cite[Theorem 12]{CCT2017} with this modification we obtain that the bound \eqref{eq:LB-test} holds with
				\begin{align}\label{eq:LB-test1}
					c(\gamma)=\exp(1-\exp(16\gamma^2))/2.
				\end{align}
			}

			We now focus on the second case $\rho > 2\sqrt{d}$. We consider only the sparse regime $s \le \rho$ since the result in the regime $s> \rho$ follows from an analogous argument. Repeating the same steps as 
			in the proof of \Cref{thm:lower-known-sigma-rho} we find:
			\begin{align}
				{\mathcal R}^*(s,d,\rho,\bar r_{\gamma}) & \ge   \inf_{\Psi}\Big[\mathbf{P}_{\mathbf{0},\bs{\S}} (\Psi=1)  + \sup_{\bs{\t} \in \tilde{\T}: \|\bs{\t}\|_2\ge 2 \gamma\s\sqrt{\phi(s, \rho^2)}} \mathbf{P}_{\bs{\t},\bs{\S}} (\Psi=0)   \Big] \\
				& \ge \inf_{\Psi}\Big[\mathbf{P}_{\mathbf{0},\bs{I}_{p \times p}} (\Psi=1)  + \sup_{\tilde{\bs{\t}} \in B_0(1, p): \|\tilde{\bs{\t}}\|_2\ge 2 \gamma\s\sqrt{r^{-1} \phi(s, \rho^2)}} \mathbf{P}_{\tilde{\bs{\t}},\bs{I}_{p \times p}} (\Psi=0)   \Big],
			\end{align}
			where in the last line the distributions are considered over a $p$-dimensional Euclidean space. 
			Since according to \eqref{eq:proof:th3} there exists an absolute constant $c' > 0$ such that $\phi(s, \rho^2) \le c' r \phi(1, p)$ we further have
			\begin{align}
				{\mathcal R}^*(s,d,\rho,\bar r_{\gamma}) \ge & \inf_{\Psi}\Big[\mathbf{P}_{\mathbf{0},\bs{I}_{p \times p}} (\Psi=1)  + \sup_{\tilde{\bs{\t}} \in B_0(1, p): \|\tilde{\bs{\t}}\|_2\ge 2 \gamma\s\sqrt{c' \phi(1, p)}} \mathbf{P}_{\tilde{\bs{\t}},\bs{I}_{p \times p}} (\Psi=0)   \Big] \\
                & \ge {\rev c(\sqrt{c'} \gamma/2)},
			\end{align}
			{\rev where the last inequality follows from the second inequality in \eqref{eq:LB-test} that we use with $s$ and $d$ taking values $1$ and $p$, respectively.}
			Combining the above lower bounds for both considered cases and using the expression for $c(\gamma)$ in \eqref{eq:LB-test1} we obtain the desired result.

\begin{acks}[Acknowledgments]
The authors would like to thank the anonymous referees, an Associate
Editor and the Editor for their constructive comments that improved the
quality of this paper.
\end{acks}

\begin{funding}
The work of Y. Wang was supported by National Key R\&D Program (No. 2022YFA1008100). The work of P. Yang was supported by Tsinghua University Dushi Program and the grant of National Natural
Science Foundation of China (NSFC) 12101353. The work of A.B. Tsybakov was supported
by the grants of French National Research Agency (ANR) ``Investissements d'Avenir'' LabEx
Ecodec/ANR-11-LABX-0047 and ANR MaLip.
\end{funding}

\bibliographystyle{imsart-nameyear} 
\bibliography{reference}       

\end{document}